\documentclass[a4paper,12pt]{amsart}
\usepackage[utf8]{inputenc}
\usepackage{epsfig}
\usepackage{amsthm,amsfonts}
\usepackage{amssymb,graphicx,color}

\newtheorem{Teo}{Theorem}[section]
\newtheorem{Cor}[Teo]{Corollary}
\newtheorem{Prop}[Teo]{Proposition}
\newtheorem{Def}[Teo]{Definition}
\newtheorem{Lem}[Teo]{Lemma}
\newtheorem{Exam}[Teo]{Example}
\newtheorem{remark}[Teo]{Remark}

\newcommand{\R}{\mathbb R}
\newcommand{\Q}{\mathbb Q}
\newcommand{\Z}{\mathbb Z}

\newcommand{\F}{\mathbb F}

\begin{document}
	
\title[Moderately discontinuous homology of real surfaces
]{Moderately discontinuous homology of real surfaces
}

\author[]{Davi Lopes Medeiros$\dagger$}\thanks{$\dagger$Research supported by the Serrapilheira Institute (grant number Serra– R-2110-39576)}
\address{Departamento de Matem\'atica, Universidade Federal do Cear\'a
	(UFC), Campus do Pici, Bloco 914, Cep. 60455-760. Fortaleza-Ce,
	Brasil}\email{profdavilopes@gmail.com}
\author[]{Jos\'e Edson Sampaio$\star$}\thanks{$\star$Research partially supported by CNPq-Brazil grant 310438/2021-7 and by the Serrapilheira Institute (grant number Serra -- R-2110-39576).}
\address{Departamento de Matem\'atica, Universidade Federal do Cear\'a
	(UFC), Campus do Pici, Bloco 914, Cep. 60455-760. Fortaleza-Ce,
	Brasil}\email{edsonsampaio@mat.ufc.br}
\author[]{Emanoel Souza}
\address{Departamento de Matem\'atica, Universidade Estadual do Cear\'a
	(UECE), Campus Itapreri, Av. Dr. Silas Munguba, 1700, Cep. 60714-903. Fortaleza-CE, Brasil}\email{emanoelfs.cdd@gmail.com}
\subjclass[2020]{51F30;   14B05;  14P10; 55N35}
\keywords{Lipschitz Geometry, Surface singularities, Snakes, Moderately discontinuous homology}

\begin{abstract}
	The Moderately Discontinuous Homology (MD-Homology, for short) was created recently in 2022 by Fern\'andez de Bobadilla at al. and it captures deep Lipschitz phenomena. However, to become a definitive powerful tool, it must be widely comprehended. 
	
	In this paper, we investigate the MD-Homology of definable surface germs for the inner and outer metrics. We completely determine the MD-Homology of surfaces for the inner metric and we present a great variety of interesting MD-Homology of surfaces for the outer metric, for instance, we determine the MD-Homology of some bubbles, snake surfaces, and horns. Furthermore, we explicit the diversity of MD-Homology of surfaces for the outer metric in general, showing how hard it is to completely solve the outer classification problem. On the other hand, we show that, under specific conditions, the weakly outer Lipschitz equivalence determines completely the MD-Homology of surfaces for the outer metric, showing that these two subjects are quite related.
\end{abstract}

\maketitle

\tableofcontents

\section*{Introduction}

Any surface germ $X$, definable in a polynomially bounded o-minimal structure over the real numbers, admits two metrics: the inner metric $d_{inn}$, where the distance between two points is the length of the shortest path connecting them inside $X$, and the outer metric $d_{out}$, where the distance between two points is their distance in the ambient space. Those two metrics define three classification problems up to bi-Lipschitz equivalence: the so-called ``inner classification problem'',  ``outer classification problem'' and ``ambient classification problem''. The first one was solved by Birbrair in \cite{birbrair2008local}, while the second and third ones are still open.

Although the outer and ambient classification problems are not completely solved yet, several important steps have been given in this direction recently. In \cite{GabrielovSouza}, Gabrielov and Souza identified finitely many building blocks for the outer Lipschitz classification of the Valette link of definable surface germs, the so-called ``normal and abnormal zones''. They also proved a weak version of the outer classification problem using the notion of ``weakly outer equivalence''. There are finitely many H\"older triangles contained in a surface $X$ which Valette link contains the normal and abnormal zones. The ones containing the abnormal zones are called snakes or bubbles, and the ones containing the normal zones are, in a certain way, well-behaved and less interesting. They characterized the abnormal zones into three possible structures, called snakes, bubbles, and non-snake bubbles.

In \cite{medeiros2023}, Medeiros and Birbrair introduced the concepts of synchronized and kneadable triangles, which is a suitable decomposition of the Valette link of surface germs that induces locally a Lipschitz ambient triviality on isolated surface germs in $\R^3$. Such a local triviality can become global for surface germs where $d_{inn}$ and $d_{out}$ are equivalent, the so-called ``Lipschitz Normally Embedded (LNE)'' sets. This result, together with the triviality obtained in \cite{jelonek2021}, shows that every LNE surface germ with isolated singularity in $\R^n$, for $n\ge 3$, $n\ne 4$, is Lipschitz ambient trivial. For $n=4$, however, Birbrair et al showed in \cite{medeiros2024} that this result is false.



The Moderately Discontinuous Homology (or MD-Homology, for short) was created in \cite{MDH2022} and was thought to be a Homology Theory that captures important metric data, since it is a complete invariant for complex analytic smooth sets and plane curve singularities. The first examples for which the MD-Homology was determined were the irreducible curves and cones. However, its potential is yet to emerge, once we believe this homology theory will be as important to Lipschitz Geometry as the ordinary homology is to topology. In particular, we believe this theory will contribute to the outer and ambient classification problems, but in order to do so, it must be widely comprehended.

In this paper, we investigate the MD-Homology groups of definable surface germs for the inner and outer metrics. We completely determine the MD-Homology groups of surfaces for the inner metric and we present a great variety of interesting MD-Homology groups of surfaces for the outer metric, as we determine the groups of some bubbles, snakes, and horns. Furthermore, we explicit the diversity of MD-Homology groups of surfaces for the outer metric in general, showing how hard it is to completely solve the outer classification problem. On the other hand, we show that, under specific conditions, the weakly outer Lipschitz equivalence determines completely the MD-Homology groups of surfaces for the outer metric, showing that these two subjects are quite related.

The paper is structured in several sections. Section \ref{sec:citacoes} is devoted to establish the basic notions on Lipschitz Geometry of germs, the main elements used to study surfaces, the concepts of MD-Homology, and its homological tools. In Section \ref{sec:degree-not-1} we calculate the MD-Homology of every surface germ when the degree is higher than 1. In Section \ref{section:MD-Surfaces-inner} we define $b$-contractions on H\"older complexes and use this to completely determine the MD-Homology of any surface germ, with the inner metric. Section \ref{sec:Examples-MDH-Outer} shows how to calculate MD-Homology groups, for the outer metric, of horns and bubble snakes, and Theorem \ref{Teo: realization for MD-Homology} gives a glimpse on how hard the problem to determine the MD-Homology is in the outer metric. Finally, in Section \ref{sec:weakouter-MDH} we study the relation between the weakly outer Lipschitz classification problem and the MD-Homology groups for the outer metric. We show that, if $X$ and $\tilde X$ are two simple $\beta$-snakes that are weakly outer equivalent and have the same structure on their corresponding subsegments and subnodal zones, then they have the same MD-Homology. This results shows that, in a broad class, weakly outer Lipschitz equivalence implies isomorphisms on MD-Homology groups for the outer metric, which is a much wider condition than the obvious outer Lipschitz equivalence.


All sets, functions and maps in this text are assumed to be definable in a polynomially bounded o-minimal structure over $\mathbb{R}$ with the field of exponents $\mathbb{F}$, for example, real semialgebraic or subanalytic. Unless the contrary is explicitly stated, we consider germs at the origin of all sets and maps.
 
\section*{Notations List}\label{sec:notations}
\begin{itemize}
	\item $\|(x_1,...,x_n)\|=(x_1^2+...+x_n^2)^{\frac{1}{2}}$;
	
	\item $\mathbb{S}_r^{n-1}(p)=\{x\in \R^n; \|x-p\|=r\}$; 
	
	\item $\mathbb{S}_r^{n-1}=\mathbb{S}_r^{n-1}(0)$ and $ \mathbb{S}^{n-1}=\mathbb{S}_1^{n-1}(0)$;
	
	\item $X_t = X \cap \mathbb{S}_{t}^{n-1}$
	
	\item $B_r^{n}(p)=\{x\in \R^n; \|x-p\|<r\}$;
	
	\item $C_{a}^{n+1}=\{ (x_1, \dots, x_n, x_{n+1}) \in \mathbb{R}^{n+1} \mid t\ge 0; \; x_1^2 + \dots + x_n^2 \le (ax_{n+1})^2 \}$
	
	\item $ -C_{a}^{n+1}=\{ - p \mid p \in C_a^{n+1} \}$ and $ C_{a}^{n+1}(r) = C_{a}^{n+1} \cap \{x_{n+1}=r\}$; 
	
	\item $U_{a}^{n+1} = \mathbb{R}^{n+1} \setminus -C_{a}^{n+1}$ and  $U_{a}^{n+1}(r) =U_{a}^{n+1} \cap \mathbb{S}_r^{n}$.
	
	\item For $X, Y\in \R^n$, $dist(X,Y):=\inf\{\|x-y\|;x\in X$ and $y\in Y\}$;
	
	\item Let $f,g\colon (0,\varepsilon)\to [0,+\infty)$ be functions. We write $f\lesssim g$ if there is a constant $f(t)\leq C g(t)$ for all $t\in (0,\varepsilon)$. We write $f\approx g$ if $f\lesssim g$ and $g\lesssim f$. We write $f\ll g$ if $\lim\limits_{t\to 0^+} \frac{f(t)}{g(t)}=0$.
	
	\item $d_{inn}$, $d_{out}$: inner and outer metric (see Definition \ref{DEF: inner, outer and normally embedded}).
	
	\item LNE: Lipschitz Normally Embedded (see Definition \ref{DEF: inner, outer and normally embedded}).
	
	\item Bi-Lipschitz map, Lipeomorphism and Lipeomorphic: see Definitions \ref{Def:Lipeo}, \ref{Def:outer-inner-lipeo}, \ref{outer-lip}.
	
	\item $V(X)$: Valette link of $X$ (see Definition \ref{Def: arc}).
	
	\item $tord$ and $itord$: tangency order on outer and inner metric (see Definition \ref{DEF: order of tangency}).
	
	\item $T(\gamma_{1},\gamma_{2})$: H\"older triangle with boundary arcs $\gamma_1, \gamma_2$ (see Definitions \ref{DEF: standard Holder triangle}, \ref{DEF: Holder triangle}).
	
	\item $T(e)$: A H\"older triangle which is a geometric presentation $\psi(Ce)$ for an H\"older complex edge $e$ (see Definition \ref{Def:geometric-complex}).
	
	\item $Lsing(X)$: Lipschitz singular arcs of $X$ (see Definition \ref{Def: generic arc of a surface}).
	
	\item $G(X)$: generic arcs of $X$ (see Definition \ref{Def: generic arc of a surface}).
	
	\item $\mu(T)$: exponent of $T$ (see Definition \ref{DEF: Holder triangle}).
	
	\item $I(T)$: set of interior arcs of $T$ (see Definition \ref{DEF: Holder triangle}).
	
	\item $\Delta(\gamma,\gamma') = \bigcup_{0\le t \ll 1}[\gamma(t),\gamma'(t)]$ (see Definition \ref{Def: Delta triangles}).
	
	\item $Abn(X)$ and $Nor(X)$: abnormal and normal zones (see Definition \ref{DEF: normal and abnormal arcs and zones}).
	
	\item (Basic/Simple) Snake: see Definitions \ref{Def:snake} and \ref{Def:Simple-basic-snake}, respectively.
	
	\item $E_{Z,Z'}^\alpha$: $\alpha$-subsegment of $Z$ (see Definition \ref{Def:subsegment}).
	
	\item $m_{X,\alpha}(\gamma)$: $\alpha$-multiplicity of $\gamma$ (see Definition \ref{Def: alpha-multiplicity}).
	
	\item $Spec(\mathcal{N})$: spectrum of the node $\mathcal{N}$ (see Definition 4.31 of \cite{GabrielovSouza}).
	
	\item $\sim_b$; $b\in (0,+\infty]$: $b$-equivalence (see Definitions \ref{Def: b-equiv of simplexes}, \ref{Def: b-equiv of simplexes2}).
	
	\item $MDH_{\bullet}^{b} (X, x_0)$ and $MDH_{\bullet}^{b} (X, x_0,d,A)$: $b$-moderately discontinuous homology group of $(X,x_0)$ (see Definition \ref{Def: MDH} and Remark \ref{Rem: MDH notation}).
	
	\item $
	\mathcal{H}_{b,\eta,d_{inn}}(X), \mathcal{H}_{b,\eta}(X)$: inner $b$-horn and $b$-horn neighborhood of amplitude $\eta$ of $X$, respectively (see Definition \ref{def:b-horn}).
	
	\item $b$-map, $b$-cover, $b$-retract, $b$-contractible: see Definitions \ref{def:bmaps}, \ref{def:conicalretract}, \ref{def:b-cover}.
	
	\item (Abstract/Geometric/Canonical) H\"older complex: see Definitions \ref{Def:abstract-complex}, \ref{Def:geometric-complex}, \ref{Def:canonical-complex}.

	\item $HC_{\mathbb{F}}$: Abstract H\"older complex (see Definition \ref{Def:abstract-complex})
\end{itemize}

\section{Preliminaries}\label{sec:citacoes}


\subsection{Normally embedded sets, outer e inner lipeomorphisms.}

\begin{Def}\label{DEF: inner, outer and normally embedded}
	Given a set $X\subset \mathbb{R}^{n}$ we can define two metrics on $X$: the \textbf{outer metric} $d_{out}(x,y)=\|x-y\|$ and the \textbf{inner metric} $d_{inn}(x,y)=\inf\{l(\alpha)\}$, where $l(\alpha)$ is the length of a rectifiable path $\alpha$ from $x$ to $y$ in $X$. A set $X \subset \R^{n}$ is \textbf{Lipschitz normally embedded} (or \textbf{LNE}, for short) if the outer and inner metrics are equivalent, i.e., there is $C\ge 1$ such that $d_{inn}(x,y) \le Cd_{out}(x,y)$ for all $x,y\in X$.
\end{Def}

\begin{remark}
	 Note that the germ of a definable set $X\subset \mathbb{R}^{n}$ is always connected and then it is path-connected. Therefore, $d_{out}(x,y) \le d_{inn}(x,y)<\infty, \forall x,y\in X$.
\end{remark}

\begin{remark}\label{Rem: pancake metric}
	The inner metric is not always definable, but one can consider an equivalent definable metric as in \cite{KurdykaOrro97}, for example, the \textbf{pancake metric} defined in \cite{LBirbMosto2000NormalEmbedding}.
\end{remark}

\begin{Def}\label{Def:Lipeo}
	Given two metric spaces $(X_1, d_1)$ and $(X_2, d_2)$, we say that a homeomorphism $\varphi: X_1 \to X_2$ is \textbf{bi-Lipschitz} (with respect to the metrics $d_1$ and $d_2$) if there exists a real number $C \ge 1$ such that:
	$$\frac{1}{C} \cdot d_1(p, q) \leq d_2(\varphi(p), \varphi(q)) \leq C \cdot d_1(p, q), \quad \forall p, q \in X_1.$$
	
	For each $C \ge 1$ that satisfies such a condition, we say that the map $\varphi$ is $C$-bi-Lipschitz. Furthermore, we say that the metric spaces $(X_1, d_1)$ and $(X_2, d_2)$ are \textbf{bi-Lipschitz equivalent} (or \textbf{lipeomorphic}) if there exists a bi-Lipschitz map $\varphi: X_1 \to X_2$.
\end{Def}

\begin{Def}\label{Def:outer-inner-lipeo}
	Given two sets $X_1 \subset \mathbb{R}^{m}$, $X_2 \subset \mathbb{R}^{n}$, an \textbf{outer bi-Lipschitz map} between $X_1$ and $X_2$ (or \textbf{outer lipeomorphism}) is a map that is bi-Lipschitz with respect to the outer metrics of $X_1$ and $X_2$. An \textbf{inner bi-Lipschitz map} between $X_1$ and $X_2$ (or \textbf{inner lipeomorphism}) is a map that is bi-Lipschitz with respect to the inner metrics of $X_1$ and $X_2$.
\end{Def}

\begin{Def}\label{outer-lip}
	Given two sets $X_1 \subset \mathbb{R}^m$, $X_2 \subset \mathbb{R}^n$, we say that $X_1$ and $X_2$ are \textbf{outer bi-Lipschitz equivalent} (or \textbf{outer lipeomorphic}) if there exists an outer lipeomorphism between $X_1$ and $X_2$. We also say that $X_1$ and $X_2$ are \textbf{inner bi-Lipschitz equivalent} (or \textbf{inner lipeomorphic}) if there exists an inner lipeomorphism between $X_1$ and $X_2$.  
\end{Def}


\subsection{Arcs, tangency orders and H\"older triangles}\label{Subsec: Holder triangles}     

\begin{Def}\label{Def: arc}
	An \textbf{arc} in $\mathbb{R}^{n}$ is a germ at the origin of a mapping $\gamma \colon [0,\epsilon) \longrightarrow \mathbb{R}^{n}$ such that $\gamma(0) = 0$. Unless otherwise specified, we suppose that arcs are parameterized by the distance to the origin, i.e., $||\gamma(t)||=t$. We usually identify an arc $\gamma$ with its image ${\rm Im}(\gamma)$ in $\mathbb{R}^{n}$. For a germ at the origin of a set $X$, the set of all arcs $\gamma \subset X$ is denoted by $V(X)$ (known as the Valette link of $X$, see \cite{valette2007link}).
\end{Def}

\begin{Def}\label{DEF: order of tangency}
	Given two arcs $\gamma_1, \gamma_2 \in V(X)$, the \textbf{tangency order} of $\gamma_{1}$, $\gamma_{2}$, denoted as $tord(\gamma_{1},\gamma_{2})$, is defined as follows:
	\begin{itemize}
		\item if $\gamma_1 \ne \gamma_2$, then $tord(\gamma_{1},\gamma_{2})$ is the exponent $q \in \mathbb{F}$ where $||\gamma_{1}(t) - \gamma_{2}(t)|| = ct^{q} + o(t^{q})$ with $c\neq 0$ (such exponent $q$ always exists, see \cite{Dries:1998});
		\item if $\gamma_1 = \gamma_2$, then $tord(\gamma_{1},\gamma_{2})=\infty$.
	\end{itemize}
	For an arc $\gamma$ and a set of arcs $Z \subset V(X)$, the tangency order of $\gamma$ and $Z$, denoted as $tord(\gamma, Z)$, is the supremum of $tord(\gamma, \lambda)$ over all arcs $\lambda \in Z$. The tangency order of two sets of arcs $Z$ and $Z'$, denoted as $tord(Z,Z')$, is the supremum of $tord(\gamma, Z')$ over all arcs $\gamma \in Z$. Similarly, we define the tangency orders in the inner metric, denoted respectively by $itord(\gamma_{1},\gamma_{2}),\; itord(\gamma, Z)$ and $itord(Z,Z')$.
\end{Def}

\begin{Def}\label{DEF: standard Holder triangle}
	For $\beta \in \mathbb{F}$, $\beta \ge 1$, the \textbf{standard} $\beta$-\textbf{H\"older triangle} $T_\beta\subset\R^2$ is the germ at the origin of the set
	\begin{equation*}\label{Formula:Standard Holder triangle}
		T_\beta = \{(x,y)\in \R^2 \mid 0\le x\le 1, \; 0\le y \le x^\beta\}.
	\end{equation*}
	The curves $\{0\le x\le 1,\; y=0\}$ and $\{0\le x\le 1,\; y=x^\beta\}$ are the \textbf{boundary arcs} of $T_\beta$.
\end{Def}

\begin{Def}\label{DEF: Holder triangle}
	A germ at the origin of a set $T \subset \mathbb{R}^{n}$ that is inner lipeomorphic to the standard $\beta$-H\"older triangle $T_\beta$ is called a $\beta$-\textbf{H\"older triangle} (see \cite{birbrair1999local}).
	The number $\beta \in \mathbb{F}$ is called the \textbf{exponent} of $T$ and is denoted by $\mu(T)$. The arcs $\gamma_{1}$ and $\gamma_{2}$ of $T$ mapped to the boundary arcs of $T_\beta$ by the lipeomorphism are the \textbf{boundary arcs} of $T$. All other arcs of $T$ are \textbf{interior arcs}. The H\"older triangle is denoted by $T=T(\gamma_{1},\gamma_{2})$ and the set of interior arcs of $T$ is denoted by $I(T)$.
\end{Def}

\begin{remark}\label{Rem: NE HT condition}
	A H\"older triangle $T$ is LNE if, and only if, $tord(\gamma,\gamma')=itord(\gamma,\gamma')$ for any two arcs $\gamma$ and $\gamma'$ of $T$ (see \cite[Theorem 2.2]{birbrair2018arc}).
\end{remark}

\begin{Def}\label{DEF: Lipschitz non-singular arc}
	Let $X$ be a surface (a two-dimensional set). An arc $\gamma \in V(X)$ is \textbf{Lipschitz non-singular} if there exists a LNE H\"older triangle $T \subset X$ such that $\gamma \in I(T)$ and $\gamma \not\subset \overline{X\setminus T}$. Otherwise, $\gamma$ is \textbf{Lipschitz singular}. It follows from pancake decomposition (see Definition \ref{Def: pancake decomposition} and Remark \ref{Rem: existence of pancake decomp}) that a surface $X$ contains finitely many Lipschitz singular arcs. The union of all Lipschitz singular arcs in $X$ is denoted by $Lsing(X)$.
\end{Def}

\begin{Exam}
	Boundary arcs or arcs which are self intersection of the surface are trivial examples of Lipschitz singular arcs. An interesting example is obtained considering $T=T_1\cup T_2$, where $T_1=T(\gamma_1,\lambda)=\{(x,y,z)\mid x\ge0,0 \le y \le x^\beta,z=0\}$ and $T_2=T(\lambda,\gamma_2)=\{(x,y,z)\mid x \ge 0,\; 0 \le y \le x^\beta,\; z = x^{\alpha-\beta}y\}$ for $\gamma_1(t)=(t,t^\beta,0),\;\gamma_2(t)=(t,t^\beta,t^\alpha)$ and $\lambda(t)=(t,0,0)$, with $\alpha,\beta \in \F$ satisfying $1\le \beta <\alpha$. Any H\"older triangle containing $\lambda$ as interior arc could not be LNE, therefore $\lambda$ is Lipschitz singular (see Example 2.11 of \cite{GabrielovSouza}).
\end{Exam}

\begin{Def}\label{DEF: non-singular Holder triangle}
	A H\"older triangle $T$ is \textbf{non-singular} if all interior arcs of $T$ are Lipschitz non-singular.
\end{Def}

\begin{Def}\label{Def: generic arc of a surface}
	Let $X$ be a surface germ with connected link. The \textbf{exponent} $\mu(X)$ of $X$ is defined as $\mu(X)=\min_{\gamma, \gamma' \in V(X)}\{ itord(\gamma,\gamma')\}$.
	A surface $X$ with exponent $\beta$ is called a $\beta$-surface. An arc $\gamma \subset X\setminus Lsing(X)$ is \textbf{generic} if $itord(\gamma,\gamma') = \mu(X)$ for all arcs $\gamma'\subset Lsing(X)$. The set of generic arcs of $X$ is denoted by $G(X)$.
\end{Def}

\begin{remark}
	Probably, for many readers, the natural way of defining $\mu(X)$ would be $\mu(X)=\inf_{\gamma,\gamma' \in V(X)}\{itord(\gamma,\gamma')\}$. However, it follows from the Arc Selecion Lemma that for a fixed $\gamma \in V(X)$, we have $\inf_{\gamma' \in V(X)}\{itord(\gamma,\gamma')\} = itord(\gamma,\gamma_0)$ for some $\gamma_0 \in V(X)$.
\end{remark}

\begin{remark}\label{Rem: generic arcs of a non-singular HT}
	If $X=T(\gamma_{1},\gamma_{2})$ is a non-singular $\beta$-H\"older triangle then an arc $\gamma\subset X$ is \textbf{generic} if, and only if, $itord(\gamma_{1},\gamma) = itord(\gamma,\gamma_{2}) = \beta$.
\end{remark}

\begin{Def}\label{Def: Delta triangles}
	Let $\gamma,\gamma' \in V(\R^n)$. We define the $\alpha$-H\"older triangle $$\Delta(\gamma,\gamma') = \bigcup_{0\le t \ll 1}[\gamma(t),\gamma'(t)],$$ where $[\gamma(t),\gamma'(t)] = \{s\gamma(t) + (1-s)\gamma'(t) \in \R^n \mid s \in [0,1]\}$ and $\alpha = itord(\gamma,\gamma')$. Notice that $\Delta(\gamma,\gamma')$ is LNE and non-singular.
\end{Def}

\begin{Def}\label{Def:weakly-outer-bi-lip}
	Let $X$ and $\tilde X$ be two $\beta$-H\"older triangles and $h \colon X \rightarrow \tilde X$ a homeomorphism. We say $h$ is \textbf{weakly outer bi-Lipschitz} when $tord(h(\gamma),h(\gamma'))>\beta$ for any arcs $\gamma$ and $\gamma'$ of $X$ if, and only if, $tord(\gamma,\gamma')>\beta$. If such a homeomorphism exists, we say that $X$ and $\tilde X$ are \textbf{weakly outer Lipschitz equivalent}.
\end{Def}

\begin{remark}
	It follows from Theorem 6.28 of \cite{GabrielovSouza} that when two $\beta$-snakes $X$ and $\tilde{X}$ are weakly outer equivalent we can write the decomposition of their Valette links as segments and nodal zones, $V(X) = \bigsqcup_{I}S_i \sqcup \bigsqcup_{J}N_j$ and $V(\tilde{X}) = \bigsqcup_{I}\tilde{S}_i\sqcup\bigsqcup_{J}\tilde{N}_j$, assuming that $S_i, \tilde{S}_i$ and $N_j,\tilde{N}_j$ are correspondent for every $i\in I$ and $j\in J$, even preserving the correspondent nodes and cluster partitions of segments (see Definition 4.32 in \cite{GabrielovSouza}).   
\end{remark}


\subsection{Pancake decomposition}\label{subsection:pancake decomposition}

\begin{Def}\label{Def: pancake decomposition}
	Let $X\subset \mathbb{R}^{n}$ be the germ at the origin of a closed set. A \textbf{pancake decomposition} of $X$ is a finite collection of closed LNE subsets $X_{k}$ of $X$ with connected links, called \textbf{pancakes}, such that $X=\bigcup_{k \in I} X_{k}$ and $$\textrm{dim}(X_{j}\cap X_{k}) < \textrm{min}(\textrm{dim}(X_{j}),\textrm{dim}(X_{k}))\quad\text{for all}\; j,k \in I.$$
\end{Def}

\begin{remark}\label{Rem: existence of pancake decomp}
	The term ``pancake'' was introduced in \cite{LBirbMosto2000NormalEmbedding}, but this notion first appeared (with a different name) in \cite{Kurdyka92} and \cite{KurdykaOrro97}, where the existence of such decomposition was established.
\end{remark}

\begin{remark}\label{Rem:pancake of holder triangle is holder triangle}
	If $X$ is a H\"older triangle then each pancake $X_k$ is also a H\"older triangle.
\end{remark}

\begin{Def}\label{Def: minimal pancak decomp}
	A pancake decomposition $\{X_{k}\}$ of a set $X$ is \textbf{reduced} if the union of any two adjacent pancakes $X_{j}$ and $X_{k}$ (such that $X_{j}\cap X_{k}\ne \{0\}$) is not LNE.
\end{Def}

\begin{remark}
	When the union of two adjacent pancakes is LNE, they can be replaced by their union, reducing the number of pancakes. Thus, a reduced pancake decomposition always exists.
\end{remark}

\subsection{Abnormal surfaces}
Some of the definitions below were first introduced in \cite{LevMendes2018}.

\begin{Def}\label{Def: zone}
	A nonempty set of arcs $Z \subset V(X)$ is a \textbf{zone} if, for any two distinct arcs $\gamma_{1}$ and $\gamma_{2}$ in $Z$, there exists a non-singular H\"older triangle $T=T(\gamma_{1},\gamma_{2}) \subset X$ such that $V(T) \subset Z$. If $Z = \{\gamma\}$ then $Z$ is a \textbf{singular zone}.
\end{Def}

\begin{Def}\label{Def: maximal zone in}
	Let $B \subset V(X)$ be a nonempty set. A zone $Z\subset B$ is \textbf{maximal in} $B$ if, for any H\"older triangle $T$ such that $V(T) \subset B$, one has either $Z\cap V(T)=\emptyset$ or $V(T) \subset Z$.
\end{Def}

\begin{remark}
	A zone could be understood as an analog of a connected subset of $V(X)$, and a maximal zone in a set $B$ is an analog of a connected component of $B$.
\end{remark}

\begin{Def}\label{Def:order of zone}
	The \textbf{order} $\mu(Z)$ of a zone $Z$ is defined as the infimum of $tord(\gamma,\gamma')$ over all arcs $\gamma$ and $\gamma'$ in $Z$. If $Z$ is a singular zone then we define $\mu(Z) = \infty$. A zone $Z$ of order $\beta$ is called a $\beta$-zone.
\end{Def}

\begin{remark}
	The tangency order can be replaced by the inner tangency order in Definition \ref{Def:order of zone}. Note that, for any arc $\gamma \in Z$, $\inf_{\gamma'\in Z}tord(\gamma,\gamma')=\inf_{\gamma'\in Z}itord(\gamma,\gamma')=\mu(Z)$. Moreover, differently from Definition \ref{Def: generic arc of a surface}, the order of a zone could not be a minimum (for a example, see Example 2.46 of \cite{GabrielovSouza}).
\end{remark}

\begin{Def}\label{Def: LNE zone}
	A zone $Z$ is LNE if, for any two arcs $\gamma$ and $\gamma'$ in $Z$, there exists a LNE H\"older triangle $T=T(\gamma,\gamma')$ such that $V(T)\subset Z$.
\end{Def}

\begin{Def}\label{DEF: normal and abnormal arcs and zones}
	A Lipschitz non-singular arc $\gamma$ of a surface germ $X$ is \textbf{abnormal} if there are two LNE H\"older triangles $T\subset X$ and $T'\subset X$ such that $T\cap T' = \gamma$ and $T\cup T'$ is not LNE.
	Otherwise $\gamma$ is \textbf{normal}. A zone is \textbf{abnormal} (resp., \textbf{normal}) if all of its arcs are abnormal (resp., normal). The sets of abnormal and normal arcs of $X$ are denoted by $Abn(X)$ and $Nor(X)$, respectively. A surface germ $X$ is called \textbf{abnormal} if $Abn(X)=G(X)$.
\end{Def}

\begin{Def}\label{Def: maximal abnormal and normal zones}
	Given an abnormal (resp., normal) arc $\gamma \subset X$, the \textbf{maximal abnormal zone} (resp., \textbf{maximal normal zone}) in $V(X)$ containing $\gamma$ is the union of all abnormal (resp., normal) zones in $V(X)$ containing $\gamma$. Alternatively, the maximal abnormal (resp., normal) zone containing $\gamma$ is a maximal zone in $Abn(X)$ (resp., $Nor(X)$) containing $\gamma$.
\end{Def}

\begin{remark}\label{max zones are unique}
	It follows from Definition \ref{DEF: normal and abnormal arcs and zones} that the property of an arc to be abnormal (resp., normal)
	is an outer bi-Lipschitz invariant: if $h:X\to X'$
	is an outer lipeomorphism then $h(\gamma)\subset X'$ is an abnormal (resp., normal) arc for any abnormal (resp., normal) arc $\gamma\subset X$.
	Since the property of an arc to be abnormal (resp., normal) is outer Lipschitz invariant,
	maximal abnormal (resp., normal) zones in $V(X)$ are also outer Lipschitz invariant: if $h:X\to X'$
	is an outer lipeomorphism then $h(Z)\subset V(X')$ is a maximal abnormal (resp., normal) zone for any maximal abnormal (resp., normal) zone $Z\subset V(X)$. Here $h:V(X)\to V(X')$ is the natural action of $h$ on arcs in $X$. 
\end{remark}

\begin{Def}\label{Def:snake}
	A non-singular $\beta$-H\"older triangle $T$ is called a $\beta$\textbf{-snake} if $T$ is an abnormal surface (see Definition \ref{DEF: normal and abnormal arcs and zones}).
\end{Def}

\begin{remark}
	It was proved in \cite{GabrielovSouza} that the Valette link of a $\beta$-snake $X$ can be decomposed into finitely many disjoint zones called \textbf{segments} and \textbf{nodal zones} (see Definition 4.19). Those zones are LNE, when the $\beta$-snake is not a spiral snake (see Proposition 4.55). Using this decomposition, Gabrielov and Souza created a combinatorial object, $W(X)$, associated with $X$. More specifically, $W(X)$ is a word, in some alphabet, say $\{x_1, x_2,\ldots \}$, satisfying two conditions (see Definition 6.6). Any word satisfying those conditions a \textbf{snake name} and they proved that, for any snake name $W$ with length $m>3$, there is a snake $X$ in $\R^{2m-1}$ such that $W = W(X)$ (see Theorem 6.23).
\end{remark}

\begin{Exam}
	In this example we illustrate how to associate a word $W(X)$ with a snake $X = T(\gamma_{1},\gamma_{2})$. Let $X$ be a snake with link as in Figure \ref{Fig. snake word}. Recall that a \textbf{node} of $X$ is the union of nodal zones with tangency orders higher than $\mu(X)$ (see Definition 4.31 of \cite{GabrielovSouza}). First, we choose an orientation for $X$, say from $\gamma_{1}$ to $\gamma_{2}$. Second, we assign letters to the nodes by moving through the link of $X$, respecting this orientation, in a way the first node encountered is assigned to the first letter of the alphabet and so on, skipping the nodes already assigned. Finally, we obtain the word associated with $X$ by traversing the link again, accordingly to the orientation, adding a letter every time we pass through a node. In this case, in the alphabet $\{x_1, x_2,\ldots \}$, we have $W(X) = [x_1x_2x_1x_3x_2x_3]$.
\end{Exam}

\begin{figure}[h!]
	\centering
	\includegraphics[width=2in]{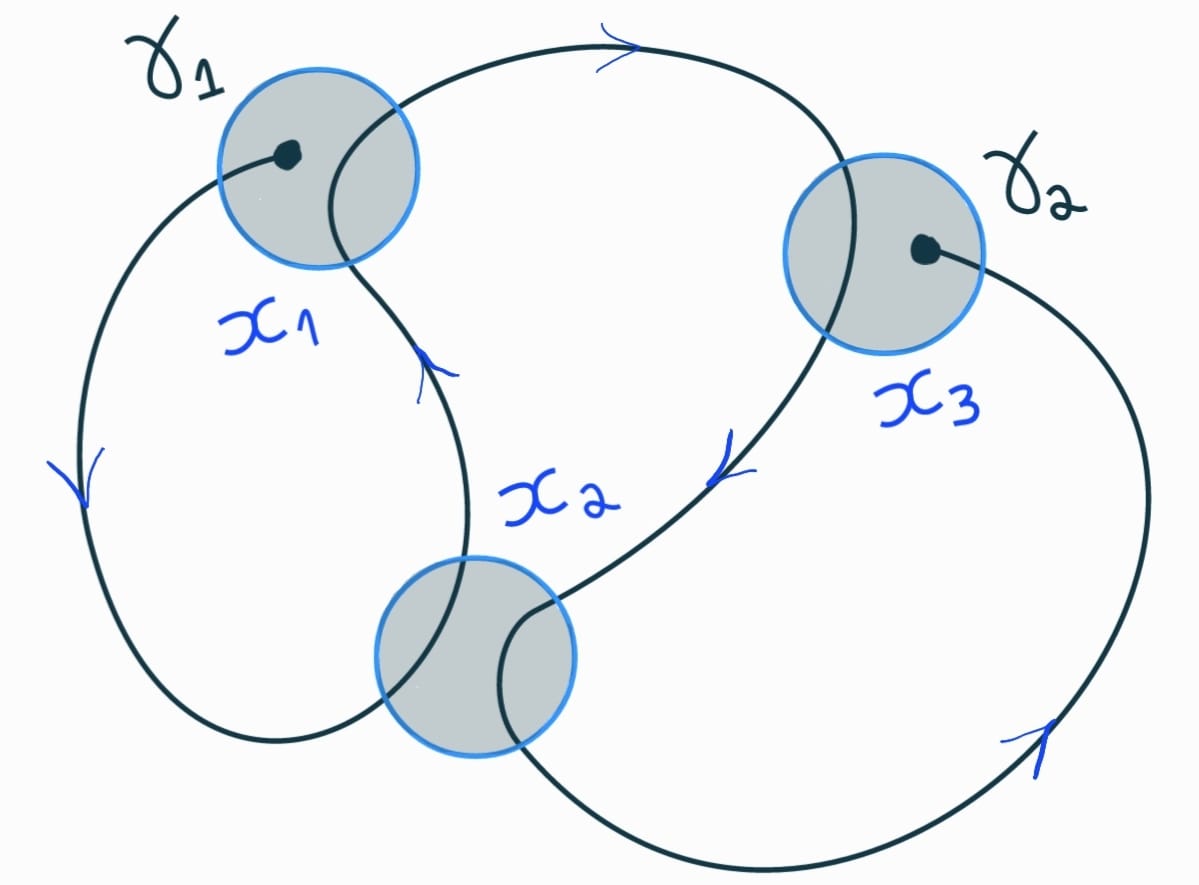}
	\caption{Example of a snake oriented from $\gamma_{1}$ to $\gamma_{2}$ with the letters $x_1$, $x_2$ and $x_3$ assigned to its three nodes, each with two nodal zones each. Points inside the shaded disks represent arcs with tangency order higher than $\beta$.}\label{Fig. snake word}
\end{figure}

\begin{Exam}\label{Exam: realization for snake names}
	Let us present a concrete example for the construction in Theorem 6.23 of \cite{GabrielovSouza}. This understanding will be fundamental for the proof of Theorem \ref{Teo: realization for MD-Homology} below. Let us start with a brief recall of the main objects in the construction. For a snake name $W=[w_1\cdots w_m]$ and exponents $\alpha > \beta \ge 1$, consider the arcs $\delta_1, \ldots, \delta_m$ and $\sigma_1,\ldots,\sigma_{m-1}$ in $\R^{2m-1} = \langle e_1,\ldots,e_{2m-1}\rangle$ defined by $\delta_1(t) = te_1$ and, for $j>1$, $\delta_j(t) = \delta_1(t) + t^\beta e_j$ if $w_i \neq w_j$, for all $i < j$, and $\delta_j(t) = \delta_{r(j)} + t^\alpha e_j$ otherwise, where $r(j) = \min\{ i < j \mid w_i = w_j\}$. Nevertheless, $\sigma_j(t) = \delta_1(t) + t^\beta e_{m + j}$ for $1\le j \le m-1$. Finally, consider $X = \bigcup_{j=1}^{m-1}T_j$, where $T_j = \Delta(\delta_j,\sigma_j)\cup \Delta(\sigma_j,\delta_{j+1})$ (see Definition \ref{Def: Delta triangles}). We have that $X$ is a $\beta$-snake and $W(X) = W$. In Figure \ref{Fig. snake realization}, black line, we have the geometric scheme of this procedure where $W = [x_1x_2x_1x_3x_2x_3]$. 
\end{Exam}

\begin{figure}[h!]
	\centering
	\includegraphics[width=2in]{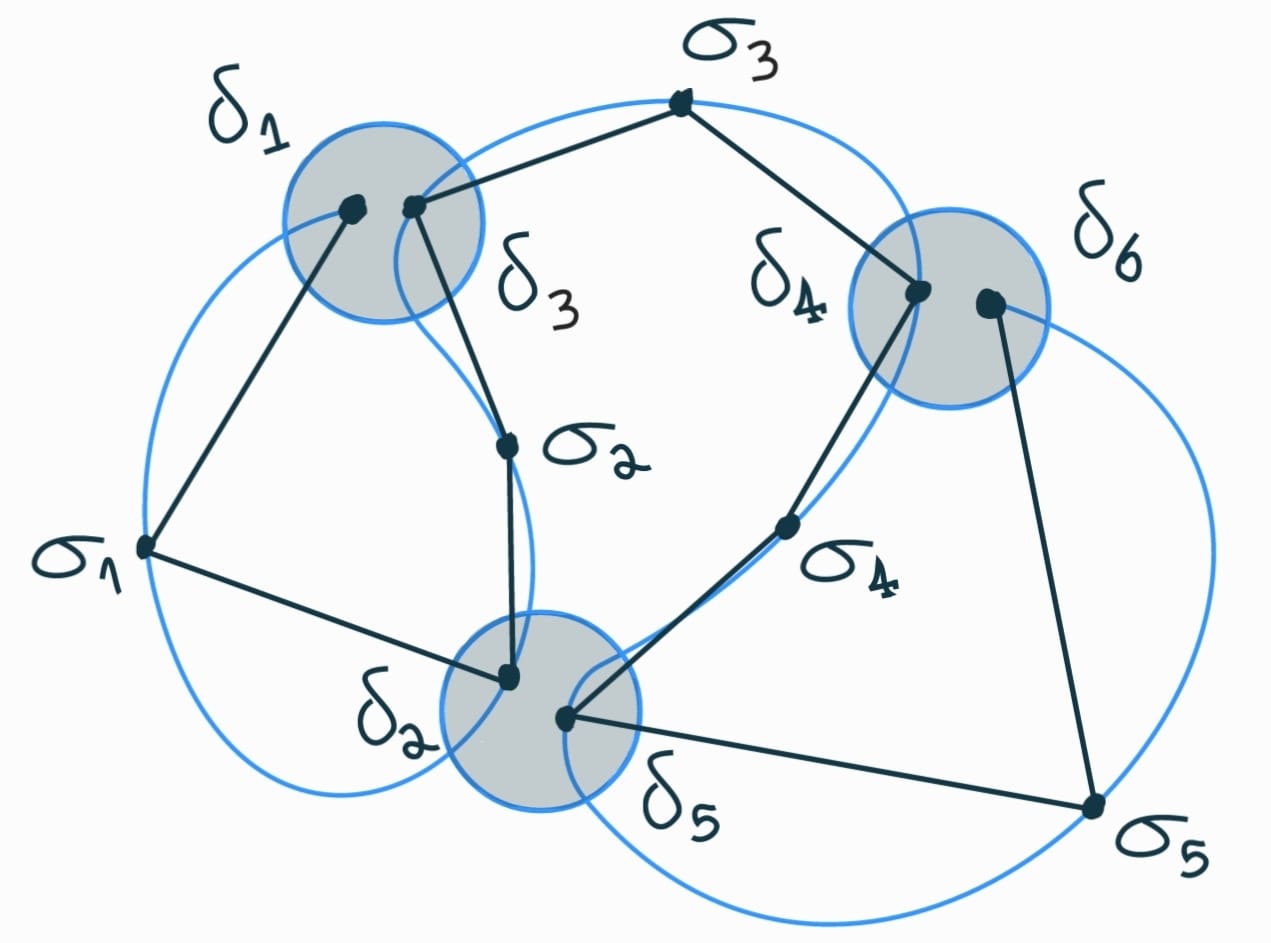}
	\caption{Link of a snake $X = \bigcup_{j=1}^{5}T_j \subset \R^{11}$ (black line) such that $W(X) = [x_1x_2x_1x_3x_2x_3]$. In blue we have the link of other snake which has the same word associated with it. Points inside the shaded disks represent arcs with tangency order higher than $\beta$.}\label{Fig. snake realization}
\end{figure}

\begin{remark}
	Despite the word associated with a snake is an interesting Lipschitz invariant combinatorial object, it ignores many geometric properties of a snake, for example, the \textbf{spectra of its nodes} (see Definition 4.31 of \cite{GabrielovSouza}).
\end{remark}

\subsection{H\"older Complex and Inner Classification of Surfaces}

\begin{Def}\label{Def:abstract-complex}
	Let $G$ a finite graph without loops (by a loop we mean an edge connecting one vertex to itself), $V(G)$ the set of its vertices and $E(G)$ the set of its edges. An {\bf abstract H\"older complex in $\mathbb{F}$} is a pair $(G, \sigma)$, where $\sigma: E(G) \to \mathbb{F}$ is a function. The set of all such abstract H\"older complexes is denoted by $HC_{\mathbb{F}}$. 
	
	Two H\"older complexes $(G_1, \sigma_1), (G_2, \sigma_2) \in HC_{\mathbb{F}}$ are \emph{isomorphic} if there is a graph isomorphism $i: G_1 \to G_2$ such that, for every $g \in E(G)$, we have $\sigma_2(i(g))=\sigma_1(g)$.
\end{Def}

\begin{Def}\label{Def:geometric-complex}
	Let $(G,\sigma) \in HC_{\mathbb{F}}$ be a H\"older complex and $X \subset \mathbb{R}^n$. For each set $S$, let $CS$ be the topological cone over $S$. We say that $X$ is a {\bf geometric H\"older complex} corresponding to $(G,\sigma)$ if the following holds:
	\begin{enumerate}
		\item there is a homeomorphism $\psi: CG \to X$ such that $\psi(a)=0$, where $a$ is the vertex of $CG$;
		\item for each $e\in E(G)$, with vertices $u$ and $v$, the set $T(e)=\psi(Ce)$ is a $\sigma(e)$-H\"older triangle with boundary arcs $\gamma_0=\psi(Cu)$ and $\gamma_1=\psi(Cv)$. Here, $Ce, Cu, Cv \subset CG$ are the subcones over $e$, $u$, $v$, respectively.
	\end{enumerate}
	The map $\psi$ is a {\bf presentation map} for $(G,\sigma)$.
\end{Def}

\begin{remark}\label{Rem:realization-complex}
	For every H\"older complex $(G, \sigma) \in HC_{\Q_{\ge 1}}$, there is $n \in \mathbb{N}$ and a closed 2-dimensional semialgebraic set germ $X \subset \R^n$ such that $X$ is a geometric H\"older complex corresponding to $(G, \sigma)$ (see Realization Theorem in \cite{Birbrair:1999realization}).
\end{remark}

\begin{Def}
	Let $G$ be a graph without loops and isolated vertices. Given $v \in V(G)$, we say that $v$ is a:
	\begin{itemize}
		\item {\bf loop vertex of $G$} if $v$ is incident with exactly two edges $e_1, e_2 \in E(G)$, with $e_1, e_2$ incident to a vertex $v_0 \in V(G)$, $v_0 \ne v$.
		\item {\bf non-critical vertex of $G$} if $v$ is incident with exactly two edges $\overline{vv_1},\overline{vv_2} \in E(G)$, with $v_1, v_2 \in V(G)$, $v_1 \ne v_2$.
		\item {\bf critical vertex of $G$} if $v$ is incident with at least three edges.
	\end{itemize}
\end{Def}

\begin{Def}
	Let $G$ be a graph with neither loops nor isolated vertices. Given a non-critical vertex $v \in V(G)$, a {\bf surgery in $v$} is the graph $\Omega_v(G)$ obtained from $G$ by removing $v$ and its two incidents edges $\overline{vv_1},\overline{vv_2}$, and connecting $v_1, v_2$ with a edge $e$. We say that $G$ is a {\bf simplified graph} if there are only critical and loop vertices in $V(G)$. We also say that the simplified graph $\tilde{G}$ is a {\bf simplification of $G$} if $\tilde{G}$ is a simplified graph and $\tilde{G}$ can be obtained from $G$ by applying a finite number (possibly none) of surgeries.
\end{Def}

\begin{Def}
	Let $(G, \sigma)$ be a H\"older Complex. Given $v \in V(G)$, consider the two following operations:
	\begin{itemize}
		\item {\bf Operation $\Omega_v$.} Let $v$ be a non-critical vertex of $G$ and $\Omega_v(G)$ a surgery in $v$, with the edges $e_1, e_2$ replaced by $e$. The operation $\Omega_v$ takes $(G, \sigma)$ to the H\"older complex $(\Omega_v(G), \sigma_v)$, where $\sigma_v(a)=\sigma(a)$ if $a \ne e$, and $\sigma_v(e)=\min \{\sigma(e_1), \sigma(e_2)\}$.
		\item {\bf Operation $\Delta_v$.} Let $v$ be a loop vertex of $G$ and $e_1, e_2 \in E(G)$ be the edges incident in $v$. The operation $\Delta_v$ takes $(G, \sigma)$ to the H\"older complex $(G, \sigma_v)$, where $\sigma_v(a)=\sigma(a)$ if $a \ne e_1, e_2$, and $\sigma_v(e_1)=\sigma_v(e_2)=\min \{\sigma(e_1), \sigma(e_2)\}$.
	\end{itemize}
	
	We say that $(G, \sigma)$ is a {\bf simplified H\"older complex} if the operations $\Omega_v$ and $\Delta_v$ cannot be applied to $(G, \sigma)$ (in particular, $G$ is a simplified graph). We also say that the simplified H\"older complex $(\tilde{G}, \tilde{\sigma})$ is a {\bf simplification of $(G, \sigma)$} if $(\tilde{G}, \tilde{\sigma})$ can be obtained by $(G, \sigma)$ by applying the operations $\Omega_v$ or $\Delta_v$ a finite number (possibly none) of times.
\end{Def}

\begin{remark}
	Every finite graph has a simplification, and the simplification is unique up to isomorphisms. Moreover, if $G_1$ and $G_2$ are two graphs homeomorphic as topological spaces in the natural graph topology, then the simplifications of $G_1$ and $G_2$ are isomorphic (see \cite{birbrair1999local}, Theorem 7.2). In the same line, every abstract H\"older complex has a simplification, and the simplification is unique up to isomorphisms (see \cite{birbrair1999local}, Proposition 7.3). Moreover, if $X$ is a geometric H\"older complex corresponding to $(G,\sigma)$, and if $(\tilde{G}, \tilde{\sigma})$ is the simplification of $(G, \sigma)$, then $X$ is a geometric H\"older complex corresponding to $(\tilde{G}, \tilde{\sigma})$ (see \cite{birbrair1999local}, Lemma 8.1).
\end{remark}

\begin{Def}\label{Def:canonical-complex}
	Let $X \subset \R^n$ be a surface, and suppose $X$ is the geometric H\"older complex corresponding to $(G,\sigma)$. If $(\tilde{G}, \tilde{\sigma})$ is the simplification of $(G, \sigma)$, the {\bf canonical H\"older complex of $X$} is defined as the H\"older complex $(\tilde{G}, \tilde{\sigma})$.
\end{Def}

\begin{Teo}[Theorem 8.2 in \cite{birbrair1999local}]\label{teo:inner-equiv-local}
	Let $X\in \R^n$, $Y \in\R^m$ be two closed surfaces. Then, $X$ and $X$ are inner lipeomorphic if, and only if, the canonical H\"older complexes of $X$ and $Y$ are isomorphic.
\end{Teo}

\subsection{Moderately discontinuous homology}\label{subsec:mdh}
Since the moderately discontinuous homology is a very recent theory, we briefly recall the main definitions of it. For more details about this theory, we refer the reader to \cite{MDH2022}. The theory works for a large class of distances, but here we only consider the outer and inner metrics.

For $n \in \Z_{\ge 0}$, let $\Delta_n \subset \mathbb R^{n+1}$ be the \textbf{standard $n$-simplex}, i.e, $$\Delta_n :=\left\{(p_0, \ldots, p_n) \in \mathbb R_{\geq 0}^{n+1}: \sum_{i=0}^n p_i = 1\right\},$$
with the orientation induced by the standard orientation of the convex hull of $\Delta_n \cup \{0\} \subset \mathbb R^{n+1}$. For $0\leq k \leq n$, define $i_n^k: \Delta_{n-1} \to \Delta_n$ as $$i_n^k(p_0, \ldots, p_{n-1}) = (p_0, \ldots, p_{k-1}, 0, p_k, \ldots, p_{n-1}).$$ Set 
$$\hat{\Delta}_n:=\{(tx, t) \in \mathbb R^{n+1}\times \mathbb R: x \in \Delta_n, t \in [0,1)\}$$ and let $\hat{j}_n^k: \hat{\Delta}_{n-1} \to \hat{\Delta}_n, (tx,t) \mapsto (ti_n^k(x),t)$. We identify $\hat{\Delta}_n$ with its germ at $(0, 0)$.

\begin{Def}  Given a definable germ $(X, x_0) \subset (\R^n,x_0)$ with a metric $d$, a \textbf{linearly vertex approaching $n$-simplex} (l.v.a. $n$-simplex) in $(X, x_0)$ is a subanalytic continuous map germ $\sigma: \hat{\Delta}_n \to (X, x_0)$ such that there is $K \geq 1$ satisfying
$$ \frac{t}{K} \leq \|\sigma(tx, t) - x_0\| \leq Kt, \, \forall (x,t) \in \hat{\Delta}_n.$$

A \textbf{linear vertex approaching $n$-chain} in $(X,x_0)$ is a finite formal sum $\sum_{i\in I} a_i \sigma_i$ where $a_i \in \mathbb{Z}$ and $\sigma_i$ is a l.v.a. $n$-simplex in $(X, x_0)$. Denote by $MDC^{pre}_n (X, x_0)$ the abelian group of $n$-chains. The \textbf{boundary of $\sigma$} is a formal sum of $(n-1)$-simplices defined as
$$\partial \sigma : = \sum_{k=0}^n (-1)^k \sigma \circ \hat{j}_n^k.$$
\end{Def}

\begin{Def} 
A \textbf{homological subdivision of $\hat{\Delta}_n$} is a finite collection $\{\rho_i\}_{i\in I}$ of injective l.v.a. map germs $\rho_i: \hat{\Delta}_n \to \hat{\Delta}_n$ such that there is a subanalytic triangulation $\alpha: |K| \to \hat{\Delta}_n$ with the following properties
\begin{itemize}
    \item $\alpha$ is compatible with faces of $\hat{\Delta}_n$;
    \item the collection $\{T_i\}$ of maximal triangles of $\alpha$ is also indexed by $I$;
    \item for any $i \in I$, $\rho_i(\hat{\Delta}_n) = T_i$ and the map $\alpha^{-1}\circ\rho_i$ takes faces of $\hat{\Delta}_n$ to faces of $K$.
\end{itemize}
The \textbf{sign of $\rho_i$}, denoted by $sgn(\rho_i)$, is defined to be $1$ if $\rho_i$ is orientation preserving, and $-1$ if it is of the opposite orientation.
\end{Def}

\begin{Def}\label{Def: b-equiv of simplexes}  
	Let $b \in (0, +\infty]$. Two $n$-simplices $\sigma_1$ and $\sigma_2$ in $MDC_n^{pre}(X, x_0)$ are called \textbf{$b$-equivalent} if 
\begin{itemize}
\item for $b < \infty$:  
$$\lim_{t\to 0} \frac{\max \{d(\sigma_1(tx,t), \sigma_2(tx,t)) \, : \, x\in \Delta_n\}}{t^b}=0.$$

\item for $b = \infty$: $\sigma_1(tx,t)= \sigma_2(tx,t), \forall x \in \Delta_n$.
\end{itemize}
The $b$-equivalence between $\sigma_1$ and $\sigma_2$ is denoted by $\sigma_1 \sim_b \sigma_2$
\end{Def}

\begin{remark}\label{Rem: b-equiv and tord of arcs}
	An arc, as in Definition \ref{Def: arc}, could be seen as a $0$-simplex. Therefore, it follows from Definition \ref{Def: b-equiv of simplexes} that two arcs $\gamma$ and $\gamma'$ are $b$-equivalent if, and only if, $tord(\gamma,\gamma')>b$. Consequently, arcs in a same node of a snake $X$ are $b$-equivalent for $b\le \mu(X)$.
\end{remark}

\begin{Def}\label{Def: b-equiv of simplexes2}   Let $b \in (0, +\infty]$.
Let  $z = \sum_{i\in I} a_i \sigma_i$ and $z' = \sum_{j\in J} b_j \tau_j$  be l.v.a. $n$-complex chains in $MDC_n^{pre}(X, x_0)$. Write $I = \bigsqcup_{k\in K}I_k$ and $J = \bigsqcup_{k\in K}J_k$ where

(a) $i_1, i_2 \in I$ belong to the same $I_k$ iff $\sigma_{i_1}\sim_b \sigma_{i_2}$.
    
(b) $j_1, j_2 \in J$ belong to the same $J_k$ iff $\tau_{j_1}\sim_b \tau_{j_2}$.

(c) for $k \in K$, $i \in I_k$, $j \in J_k$ we have $\sigma_i \sim_b \tau_j$.

Then, $z$ is called {\bf $b$-equivalent} to $z'$, denoted by $z \sim_b z'$, if for any $k$, 
$$ \sum_{i\in I_k} a_i = \sum_{j\in J_k} b_j.$$
\end{Def}



\begin{Def}\label{Def: MDH}
Let $b \in (0, +\infty]$. The \textbf{$b$-moderately discontinuous chain complex of} $(X, x_0)$ is the quotient group $MDC_{\bullet}^{b}(X, x_0) := MDC_{\bullet}^{pre}(X, x_0)/ \sim_{b}$.  Its homology is called \textbf{$b$-moderately discontinuous homology} (or \textbf{$b$-MD-Homology}, for short), denoted by $MDH_{\bullet}^{b} (X, x_0)$.
\end{Def}

\begin{remark}\label{Rem: MDH notation}
	When we want to emphasize the metric $d$ and the abelian group $A$ we write $MDH_{\bullet}^{b} (X, x_0,d,A)$ instead of $MDH_{\bullet}^{b} (X, x_0)$.
\end{remark}

\begin{Def}
	\label{def:b-horn}
	Let $X \subset \R^n$ be a subanalytic germ with metric $d$. Let $b \in (0, +\infty)$. The {\bf inner $b$-horn neighborhood of amplitude $\eta$ of $X$} is the subset 
	$$
	\mathcal{H}_{b,\eta,d_{inn}}(X):=\bigcup_{x\in X}\{z\in \R^n;d_{inn}(z,x)<\eta \|x\|^b\}.
	$$
	The {\bf $b$-horn neighborhood of amplitude $\eta$ of $X$} is the subset 
	$$
	\mathcal{H}_{b,\eta}(X):=\bigcup_{x\in X}\{z\in \R^n;\|x-z\|<\eta \|x\|^b\}.
	$$
	The $\infty$-horn neighborhood $\mathcal{H}_{\infty,\eta}(X)$ is defined to be $X$.  
\end{Def}

\begin{Def}
	\label{Def:jumpingrate}
	Let $(X,x_0)$ be a subanalytic germ. An element $b\in [1,\infty)$ is called a {\bf jumping rate} of $(X,x_0)$ if for any small enough $\varepsilon>0$ the natural homomorphism
	$$h^{b+\varepsilon, b-\varepsilon}_*: MDH^{b+\varepsilon}_*(X,x_0)\to MDH^{b-\varepsilon}_*(X,x_0)$$
	is not an isomorphism.
\end{Def}

We have the following interesting results relating moderately discontinuous homology with the singular homology:

\begin{Teo}[Theorem 10.1 in \cite{MDH2022}]
\label{Teo:link_hoologia_singular}
We have an isomorphism between the singular homology  $H_{*}(X\setminus\{x_0\},x_0)$ and $MDH^{\infty}_{*}(X,x_0)$.
\end{Teo}

\begin{Teo}[Corollary 11.12 in \cite{MDH2022}]
	\label{cor:birbrair}
	Let $X$ be a subanalytic germ at $0$ in $\R^n$. Let $b\in(0,+\infty)$ and $\eta>0$. There exists a $b'$ satisfying $b<b'$ such that the $b$-MD homology $MDH^{b}_\bullet(X,0)$ is isomorphic, for any $b''\in (b,b')$, to the singular homology of the punctured $b''$-horn neighborhood $\mathcal{H}_{b'',\eta}(X)\setminus\{0\}$. 
\end{Teo}

\begin{Teo}[Theorem 11.14 in \cite{MDH2022}] \label{Teo:finitude-geradores}
	Let $(X,x_0) \in \mathbb{R}^n$ be a subanalytic germ. Then, for each $b \in \mathbb{F}$ and $n \in \mathbb{N}$, the $\Z$-modules $MDH_n^b(X,x_0,d_{inn})$ and $MDH_n^b(X,x_0,d_{out})$ are finitely generated over $\mathbb{Z}$.
\end{Teo}

\begin{Teo}[Theorem 11.14 in \cite{MDH2022}] \label{Teo:finitude-jumping-rate}
	The set of jumping rates of $(X,x_0)$ is finite.
\end{Teo}





\subsection{$b$-maps and Mayer-Vietoris}

In this subsection we show some functoriality properties of the $b$-MD homology for a fixed $b$ by allowing a certain class of non-continuous maps, which we call $b$-maps. This makes our theory quite flexible. The discontinuities that we allow are {\em moderated} in a Lipschitz sense. This may be seen as a motivation for the name of the homology.

\begin{Def}
	\label{def:bmaps}
	Let $(X, x_0)$ and $(Y, y_0)$ be subanalytic germs. For $b \in (0, \infty)$, a \textbf{$b$-moderately discontinuous subanalytic map from $(X, x_0)$ to  $(Y, y_0)$} (or $b$-map, for short) is a finite collection $\{(C_i,f_i)\}_{i\in I}$, where $\{C_i\}_{i\in I}$ is a finite closed subanalytic  cover of $X$  and $f_i: C_i \to Y$ is a l.v.a. subanalytic map satisfying the following: for any $b$-equivalent pair of arcs $\gamma$ and $\gamma'$ contained in $C_i$ and $C_j$, respectively ($i$ and $j$ may be equal), the arcs $f_i\circ \gamma$ and $f_j\circ \gamma'$ are b-equivalent in $Y$. For $b = \infty$, a $b$-map from $X$ to $Y$ is a l.v.a. subanalytic map from $X$ to $Y$.
	
	Two $b$-maps $\{(C_i,f_i)\}_{i\in I}$ and $\{(C'_i,f'_i)\}_{i\in I'}$ are called {\bf $b$-equivalent} if for any $b$-equivalent pair of arcs $\gamma$, $\gamma'$ with Im$(\gamma) \subseteq C_i$ and Im$(\gamma') \subseteq C'_{i'}$, the arcs $f_i\circ \gamma$ and $f'_{i'}\circ \gamma'$ are b-equivalent in $Y$. We make an abuse of language and we also say that a $b$-map from $(X,x_0)$ to $(Y,y_0)$ is an equivalence class as above.

\end{Def}

\begin{remark}
	Let $\{(C_i,f_i)\}_{i\in I}$ be a $b$-map from $X$ to $Y$ and let $\{(D_j,g_j)\}_{j\in J}$ be a $b$-map from $Y$ to $Z$.  Then the composition of the two b-maps is well defined by $\{ (f_i^{-1}(D_j) \cap C_i, g_j \circ f_{i| f_i^{-1}(D_j) \cap C_i})\}_{(i,j) \in I \times J}$. Then, the category of metric subanalytic germs with $b$-maps is well defined. Moreover, the $b$-moderately discontinuous homology is invariant by isomorphisms in the category of $b$-maps.
\end{remark}

\begin{Def}
	\label{def:conicalretract}
	Let $\iota:X\hookrightarrow Y$ be a l.v.a. map of subanalytic germs which on the level of sets is an injection. A {\bf $b$-retraction} is a $b$-map $r:Y\to X$ such that 
	$r\circ\iota$ is the identity as a $b$-map. A {\bf $b$-deformation retraction} is a $b$-retraction such that $\iota\circ r$ is $b$-homotopic to the identity (for a definition of $b$-homotopy, see Definition 7.5 of \cite{MDH2022}). In those cases $X$ is called a {\bf $b$-retract} or {\bf $b$-deformation retract of $Y$}, respectively. A metric subanalytic germ is called {\bf $b$-contractible} if it admits $[0,\epsilon)$ as a $b$-deformation retract.   
\end{Def}

\begin{remark}
	\label{Rem:retracts}
	If $\iota:X\hookrightarrow Y$ admits a $b$-deformation retraction, $\iota$ induces a quasi-isomorphism of $b$-MD chain complexes. In particular, if $X$ is $b$-contractible then it has the $b$-MD-Homology of an arc (see Corollary 7.8 of \cite{MDH2022}).
\end{remark}

\begin{Def}
	\label{def:b-cover}
	Let $(X,x_0)$ be a germ and let $b \in (0, \infty]$. 
	A finite collection $\{U_i\}_{i \in I}$ of subanalytic germs in $x_0$ is called a {\bf $b$-cover of $(X,x_0)$} if it is a finite cover of $X\setminus\{x_0\}$ and, for any $i \in I$, there is a subanalytic subset $\widehat{U}_i\subseteq X$ such that 
	\begin{itemize}
		\item for any two $b$-equivalent arcs $\gamma_1,\gamma_2:[0,\epsilon)\to (X,x_0)$, if $\gamma_1$ has image in $U_i$ then $\gamma_2$ has image in $\widehat{U}_i$.  
		\item For any finite $J \subseteq I$ there is a 
		{subanalytic map $r_J:\bigcap_{i\in J}\widehat{U}_i \to \bigcap_{i \in J}  U_i$ which induces an inverse in the derived category of the  morphism of complexes:} 
		$$MDC^b_\bullet\left(\bigcap_{i\in J} U_i,x_0\right)\to MDC^b_\bullet\left(\bigcap_{i\in J}\widehat{U}_i,x_0\right).$$
	\end{itemize}
	We call the collection $\{\widehat{U}_i\}_{i \in I}$ a {\bf $b$-extension of $\{U_i\}_{i \in I}$}. A $b$-cover $\{U_i\}_{i \in I}$ of $(X,x_0)$ is said to be {\bf closed} (resp. \textbf{open}) if any of the subsets $U_i$ is closed (resp. open) in $X\setminus\{x_0\}$.
	
\end{Def}

\begin{remark}
	\label{lem:b-saturated2}
	In the terminology of Definition \ref{def:b-cover}, the following two conditions imply the respective two conditions of the definition:
	\begin{itemize}
	\item there is a $b$-horn neighborhood $\mathcal H_{b,\eta}(U_i)$ such that $\mathcal H_{b,\eta}(U_i)\cap X \subseteq \widehat{U}_i$ for any $i \in I$.
		\item For any finite $J \subseteq I$, the intersection $\bigcap_{i\in J} U_i$ is a $b$-deformation retract of $\bigcap_{i \in J} \widehat{U}_i$.
	\end{itemize}
Notice also that for $b=\infty$, any finite subanalytic cover of $X$ is a $b$-cover.	
\end{remark}






\begin{Teo}[Theorem 8.11 in \cite{MDH2022}]\label{Teo:MayerVietoris}
	Let $(X,x_0)$ be a subanalytic germ and let $\{U, V\}$ be either a closed or an open a $b$-cover of $(X,x_0)$. Then there is a Mayer-Vietoris long exact sequence as follows:
	
	\begin{align*}\label{long exact Mayer-Vietoris sequence}
		\begin{split}
			\cdots &\rightarrow MDH_{n+1}^b(X) \rightarrow  MDH_n^b(U \cap V)    \rightarrow  MDH_n^b(U) \oplus MDH_n^b(V)  \rightarrow MDH_n^b(X)             \rightarrow \cdots
		\end{split}
	\end{align*}
\end{Teo}

\section{MD-Homology of Surface Germs, Degree $n\ne 1$} \label{sec:degree-not-1}

This section is devoted to calculate the MD-homology of surface germs $(X,0)$ for $n\ne1$, for both the inner and outer metrics. We first calculate the MD-Homology of degree $n\ge 1$ for H\"older triangles, and then we use this result to calculate the MD-Homology of degree $n>1$ for any subanalytic surface germ. 

In what follows, we consider the abelian group of any MD-Homology as $\Z$ and when we omit the base point, it is understood that it is $0$. We also consider $b \in (0,+\infty)$, since the following results are immediate for $b=\infty$, by Theorem \ref{Teo:link_hoologia_singular}.


\begin{Prop}\label{Prop: inner MDH of Holder triangle}
	Let $X$ be a $\beta$-H\"older triangle. Then, $X$ is $b$-contractible for the inner metric. In particular, its MD-Homology with respect to the inner metric, for any $m\ge 1$, is $MDH_{m}^b (X, 0, d_{inn}) =
	\{0\}$.
\end{Prop}
\begin{proof}
	Let $\gamma_0 \subset X$ be an arc and consider the inclusion $\iota \colon \gamma_0 \rightarrow X$. Let $r\colon X\rightarrow \gamma_0$ be the map given by $r(x) = \gamma_0(t)$ if, and only if, $d_{out}(x,0) = t$. Notice that $r$ maps each point of $X\cap S(0,t)$ into $\gamma_0(t)$. In particular, $\iota\circ r \colon X \rightarrow X$ maps $x$ to $\gamma_0(|x|)$. Let us prove that $\iota \circ r$ is $b$-homotopic to ${\rm id}_X$, with respect to the inner metric in $X$. The result will therefore follow from Proposition 4.4 of \cite{MDH2022}, since $X$ has the same MD-Homology of an arc, by Remark \ref{Rem:retracts}. 
	
	Consider $H \colon X\times I \rightarrow X$, where $I=[0,1]$, given by $H(x,s) = (1-s)x + sr(x)$. Now, considering $\iota$, $r$ and $H$ as $b$-maps, for any two arcs $\theta,\theta' \subset X$, we must prove that if $itord(\theta,\theta') >b$ then $itord(H(\theta,s),H(\theta',s))>b$ for any $s\in I$. Indeed, as the inner tangency order is obtained from the exponent of the pancake metric (see Remark \ref{Rem: pancake metric}), it is enough to prove the case where $X$ is a pancake, i.e., $X$ is LNE. In this case, inner tangency order and the tangency order coincide. Thus, $$d_{out}(H(\theta(t),s) , H(\theta'(t),s)) = |(\theta(t)-\theta'(t)) + s(\theta'(t)-\theta(t) + r(|\theta(t)|) - r(|\theta'(t)|))|.$$
	Since arcs are assumed to be parameterized by the distance to the origin, we have $r(|\theta(t)|)=r(|\theta'(t)|)=\gamma_0(t)$. We finally get $$d_{out}(H(\theta(t),s) , H(\theta'(t),s))= |s-1||\theta'(t)-\theta(t)|.$$
	Therefore, $$itord(\theta,\theta') = itord(H(\theta,s),H(\theta',s)) .$$
\end{proof}
As an immediate consequence of Proposition \ref{Prop: inner MDH of Holder triangle}, we obtain:
\begin{Cor}\label{Cor: MDH of a NE Holder triangle}
	Let $X$ be a LNE $\beta$-H\"older triangle. Then, its MD-Homology with respect to the outer metric, for any $m\ge 1$, is $$MDH_{m}^b (X, 0, d_{out}) =
	\{0\}.$$
\end{Cor}
\begin{Teo}\label{Teo: MD-dim-ge-2}
Let $X \in \mathbb{R}^N$ be a surface germ and $m \ge 2$ be an integer. Then, for $d \in \{d_{inn}, d_{out}\}$, we have
$$MDH_m^{b}(X,0,d) \cong \{0\}$$
\end{Teo}

\begin{proof}
	Consider a pancake decomposition $\{T_i\}_{1 \le i \le n}$ for $X$, such that each $T_i$ is a LNE H\"older triangle. Such decomposition always exists, since each pancake of $X$ can be subdivided into finitely many H\"older triangles. We will prove the theorem by induction on $n$. The base case $n=1$ is trivial by Corollary \ref{Cor: MDH of a NE Holder triangle}. For the inductive step, for each $1\le i \le n$, let $\sigma_i : \hat{\Delta}_1 \to T_i$ be the 1-simplex whose image is the germ $T_i$. For each $x \in [0,1]$, let the arc $\gamma_{i,x}=\gamma_{i,x}(t) = \sigma_i(tx,t)$. For each $\eta>0$ and $i,j \in \{1,\dots,n\}$, with $i \ne j$, define
	\begin{equation*}
		I_{i,j}(\eta)=\left\{ x \in [0,1] \mid \exists \; y \in [0,1] : \gamma_{i,x} \subset \overline{\mathcal{H}_{b, \eta} (\gamma_{j,y})} \right\} \; ; \; I_{i,j}=\bigcap_{\eta>0} I(\eta).
	\end{equation*}
	The set $I_{i,j}$ can be equivalently defined as
	\begin{equation*}
		I_{i,j}=\left\{ x \in [0,1] \mid \exists \; y \in [0,1] : \lim\limits_{t\to 0^+}\frac{d(\gamma_{i,x}(t), \gamma_{j,y}(t))}{t^b}=0\right\}.
	\end{equation*}
	Therefore, $I_{i,j}$ is a definable set in $[0,1]$, which implies that $I_{i,j}$ is a finite union of points and intervals. Moreover, since $\eta < \eta^{\prime}$ implies $I_{i,j}(\eta) \subseteq I_{i,j}(\eta^{\prime})$, we have that each interval of $I_{i,j}$ is closed. We will now investigate the structure determined by $I_{i,j}$ on $T_i$ when $I_{i,j} \ne \emptyset$. In this case, $I_{i,j}$ is a union of $m_{i,j}$ closed intervals $[a_1,b_1], \dots, [a_{m_{i,j}},b_{m_{i,j}}]$ (some of them possibly degenerated), with $a_k \le b_k$ ($1\le k\le m_{i,j}$) and $b_k < a_{k+1}$ ($1\le k\le m_{i,j}-1$). 
	
	For $1\le k\le m_{i,j}$, let $T_{i,j}(k)=\left\{\bigcup_{x \in [a_k,b_k]}\gamma_{i,x}\right\}$. Each $T_{i,j}(k) \subset T_i$ is either an arc or a LNE H\"older triangle. In the latter case, each $T_{i,j}(k)$ is $b$-contractible, since $T_{i,j}(k)$ has the same $b$-homotopy type of $T_i$ and $T_i$ is $b$-contractible. Now let
	\begin{equation*}
		U=\bigcup_{k=1}^{n-1}\left(T_k\cup\left(\bigcup_{1\le p \le m_{n,k}}T_{n,k}(p)\right)\right) \; ; \; V=  T_n \cup \left( \bigcup_{k=1}^{n-1} \left( \bigcup_{1\le p \le m_{k,n}}T_{k,n}(p) \right)\right).
	\end{equation*}
	Since $\gamma_{i,x} \sim_b \gamma_{j,y}$ if and only if $x \in I_{i,j}$ and $y \in I_{j,i}$, we have that $\{U,V\}$ is a closed $b$-cover of $X$ (see Definition \ref{def:b-cover}). Moreover
	\begin{equation*}
		T_k\cup\left(\bigcup_{1\le p \le m_{n,k}}T_{n,k}(p)\right) \sim_b T_k \; ; \; T_n \cup \left( \bigcup_{1\le p \le m_{k,n}}T_{k,n}(p)\right) \sim_b T_n,
	\end{equation*}
	implying $U \sim_b T_1 \cup \dots \cup T_{n-1}$ and $V\sim_b T_n$. We also have
	\begin{equation*}
		U \cap V = \left( \bigcup_{k=1}^{n-1} \left( \bigcup_{1\le p \le m_{n,k}}T_{n,k}(p) \right)\right)\cup\left( \bigcup_{k=1}^{n-1} \left( \bigcup_{1\le p \le m_{k,n}}T_{k,n}(p) \right)\right).
	\end{equation*}
	Since each $T_{n,k}(p)$, $T_{k,n}(p)$ is $b$-contractible, $U\cap V$ is also $b$-contractible, and since the MD-Homology of degree $m-1\ge 1$ for an arc is trivial, we have $MDH_{m-1}^{b}(U \cap V) \cong \{0\}$. Now consider the following Mayer-Vietoris long exact sequence:
	\begin{equation*}
		\cdots \to MDH_m^{b}(U) \oplus MDH_m^{b}(V) \to MDH_m^{b}(X) \to MDH_{m-1}^{b}(U \cap V) \to \cdots.
	\end{equation*}
	By induction hypothesis, we have $MDH_m^{b}(U) \cong MDH_m^{b}(T_1 \cup \dots \cup T_{n-1}) \cong \{0\}$ and $MDH_m^{b}(V) \cong MDH_m^{b}(T_n) \cong \{0\}$. Then, the long exact sequence reduces to
	\begin{equation*}
		\cdots \to \{0\} \to MDH_m^{b}(X) \to \{0\} \to \cdots,
	\end{equation*}
	which implies $MDH_m^{b}(X) \cong \{0\}$, completing the inductive step.
\end{proof}

\begin{remark}\label{Rem: MD-Homology of degree 0}
	By Proposition 9.2 of \cite{MDH2022}, the $b$-MD-Homology of any subanalytic germ $(X,x_0)$ is isomorphic to $\mathbb{Z}^{r(b,X)}$, where $r(b,X)$ is the number of $b$-equivalent connected components of $X$.
\end{remark}


\section{MD-Homology for Surfaces Germs, Degree 1 (Inner Metric)}\label{section:MD-Surfaces-inner}

By the results obtained in the previous section, the most relevant information regarding Lipschitz geometry of surface germs is in this degree, for either the inner or outer metric. This section is devoted to calculate the MD-Homology of degree 1, for the inner metric.

\begin{Def}\label{contractions}
	Let $(G,\sigma) \in HC_{\mathbb{F}}$ be a H\"older complex and $b \in \mathbb{F}$, $b\ge 1$. We define a {\bf $b$-contraction of $(G,\sigma)$} as a H\"older complex $(\tilde{G},\tilde \sigma)$ obtained from $(G, \sigma)$ by applying one of the two following operations:
	\begin{itemize}
		\item {\bf Operation $A_b$:} Let $u,v \in V(G)$ connected by a cut-edge $\tilde{e} \in E(G)$ (an edge whose removal increase the number of connected components of $G$), such that $\sigma(\tilde{e}) \le b$. Let $e_1,\dots,e_p \ne \tilde{e}$ the edges incident to $u$, with $e_i$ connecting $u$ to $u_i$, for $1\le i \le p$ (possibly $p = 0$) and $f_1,\dots,f_q \ne \tilde{e}$ the edges incident to $v$, with $f_i$ connecting $v$ to $v_i$, for $1\le i \le q$ (possibly $q = 0$). Then, operation $A_b$ on $\{u,v\}$ takes $(G,\sigma)$ to the H\"older complex $(\tilde{G},\tilde \sigma)$, where $V(\tilde G) = (V(G) \setminus \{u,v\}) \sqcup \{w\}$ and
		\begin{equation*}
			E(\tilde G) = (E(G) \setminus \{\tilde{e}, e_1, \dots, e_p, f_1, \dots, f_q\}) \sqcup \{\tilde{e}_1, \dots, \tilde{e}_p, \tilde{f}_1, \dots, \tilde{f}_q\}
		\end{equation*}
		Here, $\tilde{e_i}$ is the edge connecting $w$ to $u_i$ ($1\le i \le p$) and $\tilde{f}_i$ is the edge connecting $w$ to $v_i$ ($1\le i \le q$). Moreover, $\tilde{\sigma}$ is defined as:
		$$\begin{cases}
			\tilde \sigma (e)=\sigma (e), & e \ne \tilde{e}_1,\dots,\tilde{e}_p;\tilde{f}_1,\dots,\tilde{f}_q  \\
			\tilde \sigma (\tilde{e}_i)=\sigma (e_i), & 1 \le i \le p  \\
			\tilde \sigma (\tilde{f}_i)=\sigma (f_i), & 1 \le i \le q  \\
		\end{cases}$$
		\begin{figure}[h!]
			\centering
			\includegraphics[width=6in]{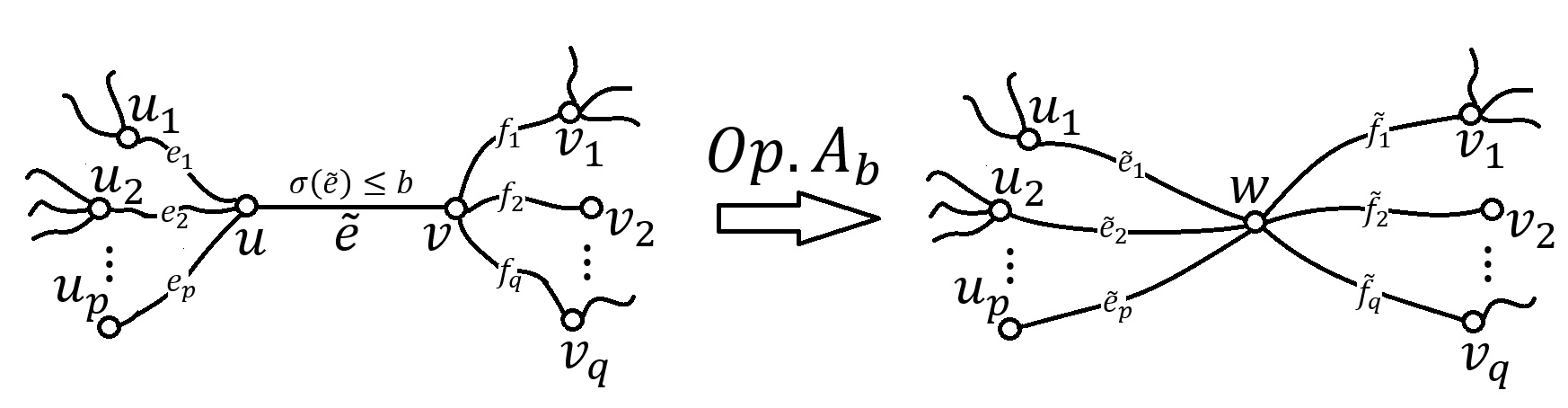}
			\caption{Diagram of Operation $A_b$.}
		\end{figure}
		
		\item {\bf Operation $B_b$:} Let $u,v \in V(G)$, $m\ge 1$ integer and let $ a_1, \dots, a_m \in E(G)$ be all the edges connecting $u$ and $v$ such that $\sigma(a_i) >b$, for $1\le i \le m$. Let $b_1, \dots, b_n \in E(G)$ be the remaining edges connecting $u$ and $v$ (possibly $n=0$), let $e_1,\dots,e_p$ be the edges incident to $u$, with $e_i$ connecting $u$ to $u_i \ne v$, for $1\le i \le p$ (possibly $p = 0$) and let $f_1,\dots,f_q$ be the edges incident to $v$, with $f_i$ connecting $v$ to $v_i \ne u$, for $1\le i \le q$ (possibly $q = 0$). Then, operation $B_b$ on $\{u,v\}$ takes $(G,\sigma)$ to the H\"older complex $(\tilde{G},\tilde \sigma)$, where $V(\tilde G) = (V(G) \setminus \{u,v\}) \sqcup \{w, w_1, \dots, w_n\}$ and
		\begin{equation*}
			E(\tilde G) = (E(G) \setminus \{a_1,\dots, a_m,b_1,\dots,b_n, e_1, \dots, e_p, f_1, \dots, f_q\}) \sqcup S
		\end{equation*}
		\begin{equation*}
			S=\{\tilde{b}_1,\tilde{b}'_1,\dots,\tilde{b}_n,\tilde{b}'_n, \tilde{e}_1, \dots, \tilde{e}_p, \tilde{f}_1, \dots, \tilde{f}_q\}
		\end{equation*}
		
		Here, $\tilde{b}_i, \tilde{b}'_i$ are the edges connecting $w$ to $w_i$ ($1\le i \le p$), $\tilde{e_i}$ is the edge connecting $w$ to $u_i$ ($1\le i \le p$) and $\tilde{f}_i$ is the edge connecting $w$ to $v_i$ ($1\le i \le q$). Moreover, $\tilde{\sigma}$ is defined as:
		$$\begin{cases}
			\tilde \sigma (e)=\sigma (e), & e \ne \tilde{b}_1,\tilde{b}'_1,\dots,\tilde{b}_n,\tilde{b}'_n,\tilde{e}_1,\dots,\tilde{e}_p;\tilde{f}_1,\dots,\tilde{f}_q  \\
			\tilde \sigma (\tilde{b}_i)=\tilde \sigma (\tilde{b}'_i)= \sigma (b_i), & 1 \le i \le n  \\
			\tilde \sigma (\tilde{e}_i)=\sigma (e_i), & 1 \le i \le p  \\
			\tilde \sigma (\tilde{f}_i)=\sigma (f_i), & 1 \le i \le q  \\
		\end{cases}$$
		\begin{figure}[h!]
			\centering
			\includegraphics[width=6in]{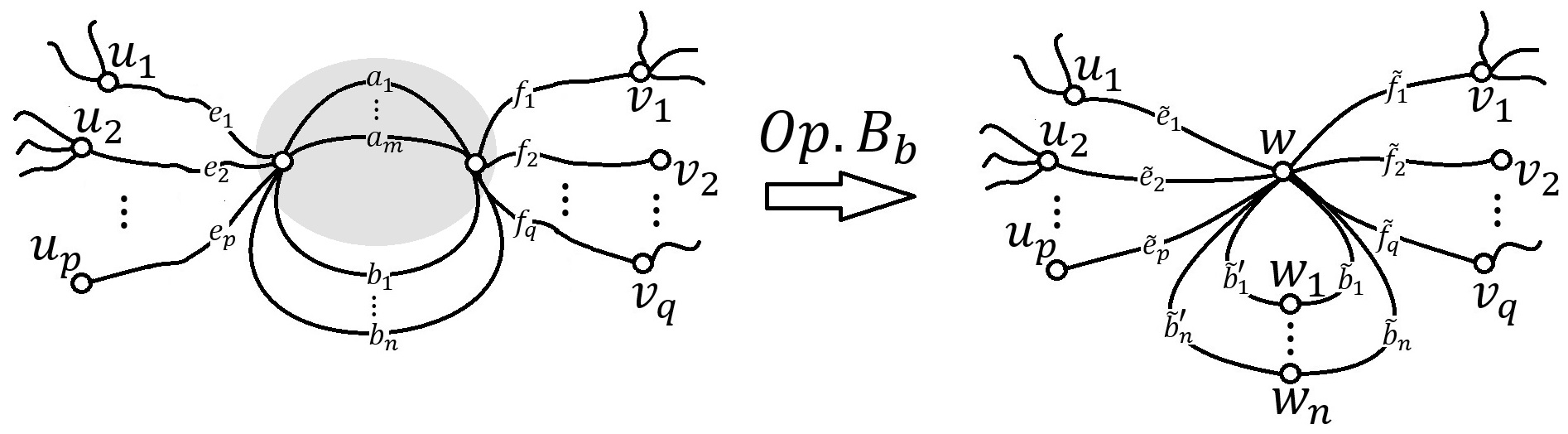}
			\caption{Diagram of Operation $B_b$. Edges inside the shaded disk represent H\"older triangles with exponent higher than $b$.}
		\end{figure}
		
	\end{itemize}
	In both operations, we say that $w$ is the {\bf merge of $u$ and $v$}. We also say that $(G,\sigma)$ is {\bf $b$-reduced} if it is not possible to perform a $b$-contraction on $(G,\sigma)$.
\end{Def}

\begin{remark}
	It is clear that each $b$-contraction $(\tilde{G},\tilde{\sigma})$ of a H\"older complex $(G,\sigma) \in HC_{\mathbb{F}}$ is in fact a well-defined abstract H\"older complex in $HC_{\mathbb{F}}$.
\end{remark}

\begin{Lem}\label{Lema-existência-reducao}
	Let $(G_0,\sigma_0) \in HC_{\mathbb{F}}$ be a canonical H\"older complex and $b \in \mathbb{F}$, $b \ge 1$. Then, there is $n\ge 0$ integer such that $(G_n,\sigma_n)$ is $b$-reduced and, for $i=1,\dots,n$, $(G_i,\sigma_i)$ is a $b$-contraction of $(G_{i-1},\sigma_{i-1})$.
\end{Lem}

\begin{proof}
	For each $(G,\sigma) \in HC_{\mathbb{F}}$, let $f_A(G)$ be the number of cut-edges of $G$ and $f_B(G)$ be the number of edges $e\in E(G)$ such that $\sigma(e)>b$. Notice first that, if $(\tilde G, \tilde \sigma)$ is any $b$-contraction of $(G,\sigma)$, then $f_A(\tilde G) +f_B(\tilde G) < f_A(G) +f_B(G)$. Since $f_A(G) +f_B(G)\ge 0$ for every $(G,\sigma) \in HC_{\mathbb{F}}$, we conclude that $n$ is finite.

\end{proof}

\begin{remark}\label{Rem:unicidade-b-contracao}
	The $b$-contraction obtained in Lemma \ref{Lema-existência-reducao} is unique. However, this fact will not be used in this paper. Hence, since its proof is quite long, we have chosen to omit this fact.
\end{remark}

\begin{Def}
	Let $(G,\sigma) \in HC_{\mathbb{F}}$ and $(G_n, \sigma_n)$ as in Lemma \ref{Lema-existência-reducao}. A {\bf $b$-reduced H\"older complex} of $(G,\sigma)$ is any H\"older complex isomorphic to $(G_n, \sigma_n)$.
\end{Def}

\begin{Prop}\label{reducao-invariante-homologia}
	Let $(G,\sigma) \in HC_{\mathbb{F}}$ and $(\tilde{G},\tilde{\sigma})$ a $b$-contraction of $(G,\sigma)$. If $X$ and $\tilde X$ are geometric H\"older complex germs corresponding to $(G,\sigma)$ and $(\tilde{G},\tilde{\sigma})$, respectively, then
	$$MDH_1^{b}(X,0,d_{inn}) \cong MDH_1^{b}(\tilde X,0,d_{inn})$$
\end{Prop}

\begin{proof}
	Let $\psi, \tilde{\psi}$ be presentation maps for $(G,\sigma)$ and $(\tilde{G},\tilde{\sigma})$, respectively. In what follows, we consider the same elements and notations as in Definition \ref{contractions}, for example, $\tilde e$ is the cut-edge in Operation $A_b$, $a_1,\dots,a_m$ are the edges connecting $u$ and $v$ with exponent higher than $b$ in Operation $B_b$, $w$ is the merge of $u$ and $v$ and so on. Let also $\gamma_0=\psi(Cu)$, $\gamma_1=\psi(Cv)$, $\tilde{\gamma}=\tilde{\psi}(Cw)$. We have two cases to consider:
	
	{\bf Case 1:} if $(\tilde{G},\tilde{\sigma})$ was obtained from $(G,\sigma)$ from operation $A_b$ on $\{u,v\}$, then let $\gamma \in G(T(\tilde e))$ be a generic arc, $G_0 \subset G\setminus \{\tilde e\}$ be the connected component containing $u$ and $G_1 = (G\setminus\{\tilde e\})\setminus G_0$ (we have $v \in G_1$, because $\tilde e$ is a cut-edge). Choose $\eta>0$ such that, for any $t>0$ small enough, the length of the curves $T(\gamma_0,\gamma)_t, T(\gamma_1,\gamma)_t$ are greater than $\eta t^{\sigma(\tilde e)}$ and the length of each curve $(T(e))_t$ is greater than $\eta t^{\sigma(e)}$,	for each $e \in E(G)\cup E(\tilde{G})$. If $\tilde{E}=\{e_1,\dots,e_p,f_1,\dots,f_q\}$, define
	\begin{equation*}
		U=X\cap \mathcal{H}_{b,\eta,d_{inn}}\left(\left(\bigcup_{e \in E(G_0)}T(e) \right) \cup T(\gamma_0,\gamma) \right)
	\end{equation*}
	\begin{equation*}
		V=X\cap \mathcal{H}_{b,\eta,d_{inn}}\left(\left(\bigcup_{e \in E(G_1)}T(e) \right) \cup T(\gamma_1,\gamma) \right)
	\end{equation*}
	\begin{equation*}
		\tilde U=\tilde{X} \cap \mathcal{H}_{b,\eta,d_{inn}}\left(\bigcup_{e \in (E(\tilde{G})\setminus \tilde{E})\cup\{e_1,\dots,e_p\}}T(e) \right)
	\end{equation*}
	\begin{equation*}
	\tilde V=\tilde{X} \cap\mathcal{H}_{b,\eta,d_{inn}}\left(\bigcup_{e \in (E(\tilde{G})\setminus \tilde{E})\cup\{f_1,\dots,f_q\}}T(e) \right)
	\end{equation*}
	Since $\tilde e$ is a cut-edge and $\sigma(\tilde e) \le b$, by the choice of $\eta$ we have $U \cap V = X \cap \mathcal{H}_{b,\eta,d_{inn}}(\gamma)$, which is a $b$-H\"older triangle and therefore $b$-contractible, by Proposition \ref{Prop: inner MDH of Holder triangle}. Analogously, $\tilde U \cap \tilde V = \tilde X \cap \mathcal{H}_{b,\eta, d_{inn}}(\tilde \gamma)$ is $b$-equivalent to a finite union of $b$-H\"older triangles, where any two of them have intersection equal to $\tilde \gamma$. This implies that $\tilde \gamma$ is a $b$-deformation retract of $\tilde U \cap \tilde V$ and therefore
	\begin{equation*}
		MDH_1^b(U\cap V, d_{inn}) \cong MDH_1^b(\tilde U\cap \tilde V, d_{inn}) \cong \{0\}
	\end{equation*}
	\begin{equation*}
		MDH_0^b(U\cap V,d_{inn}) \cong MDH_0^b(\tilde U\cap \tilde V, d_{inn})\cong \mathbb{Z}
	\end{equation*}	
	We also have that $U_0=\bigcup_{e \in E(G_0)}T(e)$ is a $b$-deformation retract of $U$, $V_0=\bigcup_{e \in E(G_1)}T(e)$ is a $b$-deformation retract of $V$, $\tilde U_0 = \bigcup_{e \in (E(G_0)\setminus\{e_1, \ldots,e_p\})\cup\{\tilde e_1, \ldots, \tilde e_p\}}T(e)$ is a $b$-deformation retract of $U$ and $\tilde V_0 = \bigcup_{e \in (E(G_1)\setminus\{f_1, \ldots,f_q\})\cup\{\tilde f_1, \ldots, \tilde f_q\}}T(e)$ is a $b$-deformation retract of $V$. By Theorem \ref{teo:inner-equiv-local}, $U_0$ is inner lipeomorphic to $\tilde{U}_0$ and $V_0$ is inner lipeomorphic to $\tilde{V}_0$. Therefore
	$$MDH_m^b(U, d_{inn}) \cong MDH_m^b(U_0, d_{inn}) \cong MDH_m^b(\tilde U_0, d_{inn}) \cong MDH_m^b(\tilde{U}, d_{inn})$$
	$$MDH_m^b(V, d_{inn}) \cong MDH_m^b(V_0, d_{inn}) \cong MDH_m^b(\tilde V_0,  d_{inn}) \cong MDH_m^b(\tilde{V}, d_{inn})$$
	
	\begin{figure}[h!]
		\centering
		\includegraphics[width=6in]{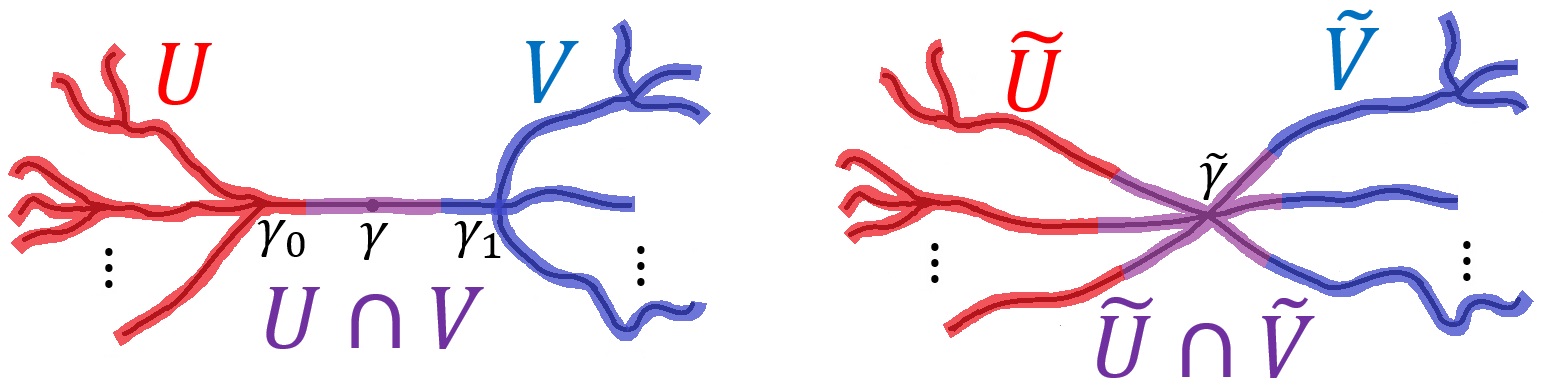}
		\caption{Proof of Proposition \ref{reducao-invariante-homologia}, Case 1.}
	\end{figure}
	
	By Theorem \ref{Teo:MayerVietoris}, we have the following Mayer-Vietoris exact sequences:
	$$\cdots \to \{0\} \to MDH_1^{b}(U) \oplus MDH_1^{b}(V) \to MDH_1^{b}(X) \to MDH_{0}^{b}(U\cap V) \to \mathbb{Z}\oplus \mathbb{Z} \to \cdots $$
	$$\cdots \to \{0\} \to MDH_1^{b}(\tilde U) \oplus MDH_1^{b}(\tilde V) \to MDH_1^{b}(\tilde X) \to MDH_{0}^{b}(\tilde U\cap \tilde V) \to \mathbb{Z}\oplus \mathbb{Z} \to \cdots $$
	
	By the Five Lemma, we conclude that $MDH_1^{b}(X) \cong MDH_1^{b}(\tilde X)$.
	
	{\bf Case 2:} if $(\tilde{G},\tilde{\sigma})$ was obtained from $(G,\sigma)$ from operation $B_b$ on $\{u,v\}$, then consider $M> {\rm max}\{\sigma(e) \mid e \in E(G)\}$ and for each $1\le i \le n$, $1\le j \le p$, let $\gamma_{b_i} \in V(T(b_i))$, $\gamma_{e_j} \in V(T(e_j))$ be arcs such that
	$itord(\gamma_0,\gamma_{b_i})=itord(\gamma_0,\gamma_{e_j})=M$. Let also 
	$$X^{\prime}=X\setminus\left(\left(\bigcup_{i=1}^{n}T(\gamma_0,\gamma_{b_i})\right)\cup \left(\bigcup_{j=1}^{p}T(\gamma_0,\gamma_{e_j})\right)\cup \left(\bigcup_{l=1}^{m}T(a_l)\right)\right)$$
	
	Since $\gamma_0 \sim_{b} \gamma_1 \sim_{b} \gamma_{b_i} \sim_{b} \gamma_{e_j}$, we have $X \sim_{b} X^{\prime}$. Moreover, if $\tilde T_{b_i} = T(\gamma_{b_i},\gamma_1)$, $\tilde T_{e_j} = T(\gamma_{e_j},\gamma_1)$ are H\"older triangles with exponent higher than $M$ (by Remark \ref{Rem:realization-complex}, such triangles exists if $X^{\prime}$ is embedded in $\R^N$, for $N$ big enough), and if $X^{\prime \prime}=X^{\prime}\cup ((\cup_{i=1}^{n}\tilde T_{b_i})\cup (\cup_{j=1}^{p}\tilde T_{e_j}))$, then $X^{\prime} \sim_b X^{\prime \prime}$. But the canonical H\"older complexes associated with $X^{\prime \prime}$ is isomorphic to $(\tilde G, \tilde \sigma)$ by construction, because $\gamma_1$ corresponds to $w$ and the H\"older triangles $\tilde T_{b_i}$, $\tilde T_{e_j}$ have exponent higher than any adjacent H\"older triangle. Hence $MDH_1^{b}(X) \cong MDH_1^{b}(X^{\prime \prime}) \cong MDH_1^{b}(\tilde X)$.
	
	\begin{figure}[h!]
		\centering
		\includegraphics[width=5in]{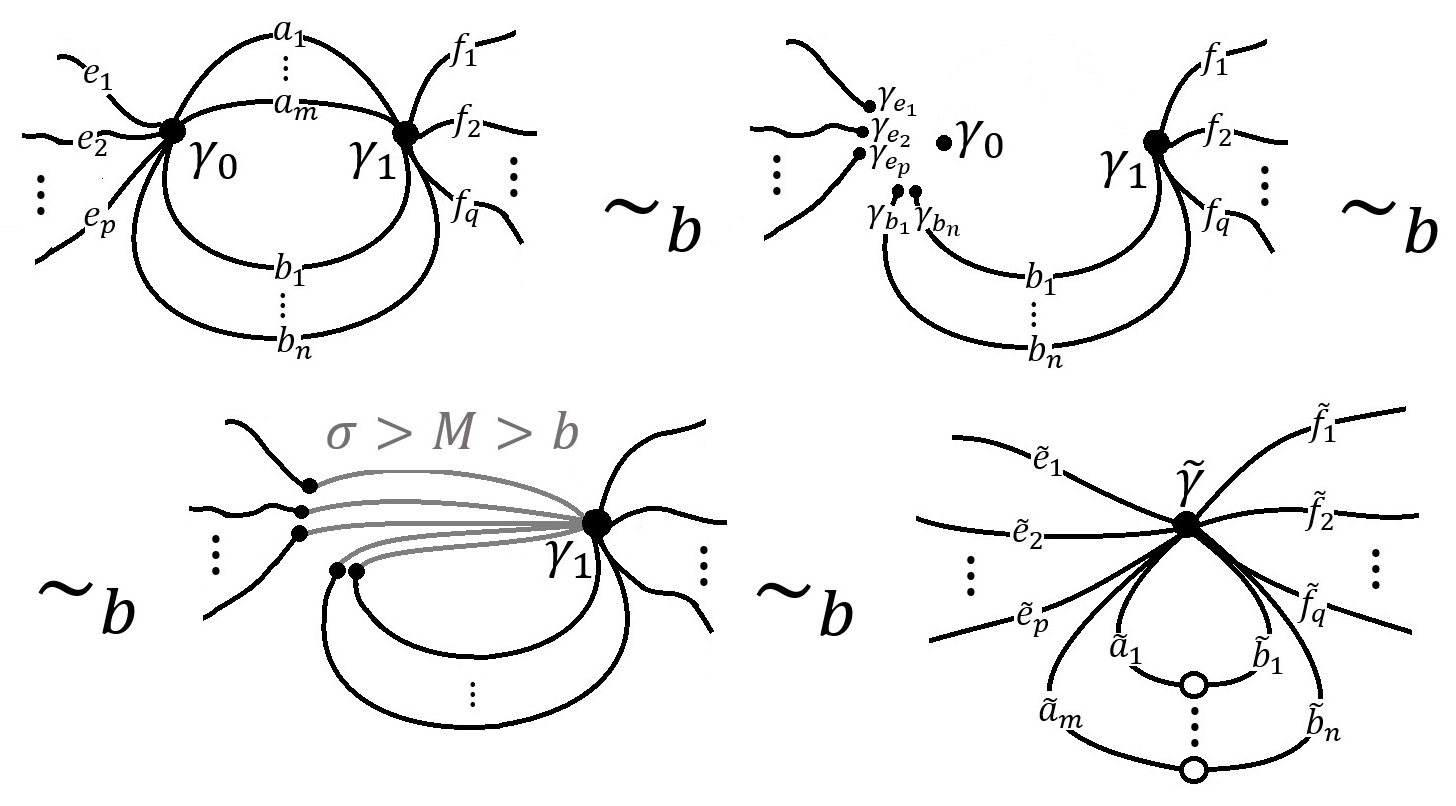}
		\caption{Proof of Proposition \ref{reducao-invariante-homologia}, Case 2.}
	\end{figure}
\end{proof}

\begin{Teo}[MD-Homology of Degree 1, Inner Metric]\label{Teo: formula for inner HMD}
	Let $(X,0) \subset (\mathbb{R}^n,0)$ be a surface germ and let $(G,\sigma)$ its canonical H\"older complex. For $b \in \mathbb{F}$, $b \ge 1$, let $(G_b,\sigma_b)$ be a $b$-reduced H\"older complex of $(G,\sigma)$. If $G_b$ has $k$ connected components, $v$ vertices and $e$ edges, then:
	$$MDH_1^{b}(X,0,d_{inn},\mathbb{Z}) \cong \mathbb{Z}^{e-v+k}$$
\end{Teo}

\begin{proof}
	Notice first that if $G_b$ is a union of connected graphs $G_1,\ldots,G_k$, with $G_i$ having $v_i$ vertices and $e_i$ edges, for each $1\le i \le k$, then $e_i -v_i+1\ge 0$, with equality if, and only if, $G_i$ is a tree. Since $v=v_1+\dots+v_k$ and $e=e_1+\dots+e_k$, we have $e-v+k=\sum_{i=1}^{k} (e_i-v_i+1) \ge 0$, with equality if, and only if, $G_b$ is a union of trees. In particular, equality occurs if, and only if, $G_b$ does not have any cycle.
	
	We will prove the theorem by induction on $e-v+k$. For the base case $e-v+k=0$, $G_b$ is a union of trees, which implies that each edge of $G_b$ is a cut-edge. Since $(G_b,\sigma_b)$ is $b$-reduced, each edge $\tilde e$ of $G_b$ satisfies $\sigma_b(\tilde e)\le b$, and therefore this edge cannot be a cut-edge (otherwise operation $A_b$ can be done). Then $G_b$ is a union of isolated vertices and therefore $MDH_1^{b}(X,0,d_{inn},\mathbb{Z}) \cong \{0\} \cong \mathbb{Z}^{e-v+k}$.
	
	For the inductive step, let $\psi$ be a presentation map for $(G_b,\sigma_b)$. Since $e-v+k>0$, $G_b$ has a cycle (otherwise, $G_b$ is a union of trees, and this is the base case), and therefore there is an edge $e_0 \in E(G_b)$ in a cycle of $G_b$. Suppose $T(e_0)=T(\gamma, \gamma')$. Consider $\eta>0$ such that, for any edge $\tilde e \in E(G_b)$ and $t>0$ small enough, the length of the curve $T(\tilde e)_t$ is bigger than $2\eta t^b$ ($\eta$ exists, since $\sigma_b(\tilde e)\le b$). Define
	\begin{equation*}
	U=X\cap \mathcal{H}_{b,\eta,inn}\left( X \setminus T(e_0) \right) \; ; \; V=X\cap \mathcal{H}_{b,\eta,inn}\left(T(e_0) \right).
	\end{equation*}
	We have that $\{U,V\}$ is a open $b$-cover of $X$. By the choice of $\eta$, $\overline{U\cap V}$ have $\gamma\cup \gamma'$ as a $b$-deformation retract, $\overline{U}$ has $X\setminus T(e_0)$ as a $b$-deformation retract and $\overline{V}$ have $\gamma$ (or $\gamma'$) as a $b$-deformation retract. Thus, for every $m\ge 0$ 
	\begin{equation*}
		MDH_{m}^{b}(U,0,d_{inn},\mathbb{Z}) \cong MDH_{m}^{b}(\psi(C(G_b\setminus \{e_0\}))),0,d_{inn},\mathbb{Z}),
	\end{equation*}
	\begin{equation*}
		MDH_{m}^{b}(V,0,d_{inn},\mathbb{Z}) \cong MDH_{m}^{b}(T(e_0),0,d_{inn},\mathbb{Z})\cong MDH_{m}^{b}(\gamma,0,d_{inn},\mathbb{Z}),
	\end{equation*}
	\begin{equation*}
		MDH_{m}^{b}(U\cap V,0,d_{inn},\mathbb{Z}) \cong MDH_{m}^{b}(\gamma,0,d_{inn},\mathbb{Z})\oplus MDH_{m}^{b}(\gamma^{\prime},0,d_{inn},\mathbb{Z}).
	\end{equation*}
	On the other hand, by Theorem \ref{Teo:MayerVietoris}, we have the following Mayer-Vietoris exact sequence:
	\begin{equation*}
		\cdots \to MDH_{1}^{b}(U\cap V) \to MDH_1^{b}(U) \oplus MDH_1^{b}(V) \to MDH_1^{b}(X) \to 
	\end{equation*}
		\begin{equation*}
		\to MDH_{0}^{b}(U\cap V) \to MDH_{0}^{b}(U) \oplus MDH_{0}^{b}(V) \to MDH_0^{b}(X) \to 0.
	\end{equation*}	
	By induction hypothesis, we have $MDH_1^{b}(U) \cong \mathbb{Z}^{(e-1)-v+k}$, because operation $A_b$ leaves $(e-v+k)$ invariant and $G_b\setminus \{ e_0\}$ does not have any edge $e$ such that $\sigma(e)>b$, otherwise operation $B_b$ would be applicable in $G_b\setminus \{ e_0\}$, a contradiction. Since $U\cap V$ and $V$ have the same MD-Homology of arcs, we also have $MDH_1^{b}(U\cap V) \cong MDH_1^{b}(V) \cong \{0\}$,  $MDH_0^{b}(V) \cong \mathbb{Z}$ and $MDH_0^{b}(U\cap V) \cong \mathbb{Z} \oplus \mathbb{Z} \cong \mathbb{Z}^2$. Since $G_b$ and $G_b\setminus \{ e_0\}$ have $k$ connected components, we finally have $MDH_0^{b}(U)\cong MDH_0^{b}(X) \cong \mathbb{Z}^k$. Thus, the exact sequence reduces to
	\begin{equation*}
		\cdots \to 0 \stackrel{f_0}{\rightarrow} \mathbb{Z}^{e-v+k-1} \stackrel{f_1}{\rightarrow} MDH_1^{b}(X)  \stackrel{f_2}{\rightarrow} \mathbb{Z}^2 \stackrel{f_3}{\rightarrow} \mathbb{Z}^{k+1}  \stackrel{f_4}{\rightarrow} \mathbb{Z}^k \to 0.
	\end{equation*}
	Since ${\rm dim}({\rm ker}(f_1))={\rm dim}({\rm im}(f_0))=0$, if $MDH_1^{b}(X) \cong \mathbb{Z}^N$, for some $N \in \mathbb{Z}_{\ge 0}$, by exactness of the sequence we obtain
	\begin{equation*}
		{\rm dim}({\rm im}(f_4))=k \Rightarrow {\rm dim}({\rm im}(f_3))={\rm dim}({\rm ker}(f_4))=(k+1)-k=1 =1 \Rightarrow
	\end{equation*}
	\begin{equation*}
		{\rm dim}({\rm im}(f_2))= {\rm dim}({\rm ker}(f_3))=2-1=1  \Rightarrow {\rm dim}({\rm im}(f_1))={\rm dim}({\rm ker}(f_2))=N-1 \Rightarrow 
	\end{equation*}
		\begin{equation*}
		e-v+k-1={\rm dim}({\rm ker}(f_0))={\rm dim}({\rm im}(f_1))=N-1\Rightarrow N=e-v+k.
	\end{equation*}
	Therefore, $MDH_1^{b}(X,0,d_{inn},\mathbb{Z}) \cong \mathbb{Z}^{e-v+k}$ and the result follows by induction.
	
\end{proof}

\section{MD-Homology for Surfaces Germs, Degree 1 (Outer Metric)}\label{sec:Examples-MDH-Outer}

In this section, we begin the study of the MD-Homology of surfaces, for the outer metric. We begin this section by proving the Gluing Triangles Lemma, then applying it to determine the MD-Homology for bubble snakes and horns. Even in this simple case, the Realization Theorem \ref{Teo: realization for MD-Homology} shows that the MD-Homology can be as complicated as we want.

\begin{Lem}\label{Lem: gluing-triangles}
	Let $\beta \in (0,+\infty)$, $\tilde{\beta} \in (\beta,+\infty)$ and let $X=T(\gamma_0,\gamma_1)$, $T=T(\gamma_1,\gamma_2)$ be H\"older triangles such that $T$ is a LNE $\beta$-H\"older triangle and
	\begin{enumerate}
		\item[(i)] for every arc $\gamma \subset T$ such that $tord(\gamma, \gamma_1)=\beta$, we have $tord(\gamma,X) = {\rm min}\{\tilde \beta, tord(\gamma, \gamma_2)\}$ (in particular, $tord(\gamma_2,X)=\tilde \beta$);
		\item[(ii)] for every arc $\gamma \subset T$ such that $tord(\gamma, \gamma_2)=\beta$, we have $tord(\gamma,X) = tord(\gamma, \gamma_1)$ (in particular, $tord(\gamma,X)=\beta$, for every $\gamma \in G(T)$).
	\end{enumerate}
	Then we have the following isomorphism of MD-Homology:
	$$MDH_1^{b}(X\cup T,0,d_{out})\cong 
	\begin{cases}
		MDH_1^{b}(X,0,d_{out}), & b< \beta  \\
		MDH_1^{b}(X,0,d_{out}) \oplus \mathbb{Z}, & \beta \le b < \tilde{\beta}  \\
		MDH_1^{b}(X,0,d_{out}), & \tilde{\beta} \le b \\
	\end{cases}$$
\end{Lem}

\begin{proof}
	Consider first the case $b < \beta$. For every arc $\gamma \subset T$ we have $tord (\gamma_1, \gamma) \ge \beta >b$. This implies that, if $\sigma_T$, $\sigma_{\gamma_1}$ are simplexes whose images germs are respectively $T, \gamma_1$, then $\sigma_T \sim_b \sigma_{\gamma_1}$. Therefore, $MDH_1^{b}(X\cup T)\cong MDH_1^{b}(X\cup \gamma_1) \cong MDH_1^{b}(X).$
	
	For the case $\beta \le b$, given two $b$-equivalent arcs $\tilde{\gamma} \subset X$, $\gamma \subset T$, by the $b$-equivalence we have $\beta \le b < tord(\tilde \gamma, \gamma)$. Then, $\gamma \notin G(T)$, otherwise condition (ii) would imply $tord(\tilde \gamma, \gamma) \le tord(\gamma,X)=\beta$. If $tord(\gamma, \gamma_1) = \beta$, then by condition (i) we have
	\begin{equation}\label{eq:1}
		b < tord(\gamma, \tilde \gamma) \le tord(\gamma, X) \le tord(\gamma, \gamma_2).
	\end{equation}
	If $tord(\gamma, \gamma_2) = \beta$, then by condition (ii) we have
	\begin{equation}\label{eq:2}
		b< tord(\gamma, \tilde \gamma) \le tord(\gamma,X) = tord(\gamma, \gamma_1).
	\end{equation}
	Therefore, $\gamma \sim_b \gamma_1$ or $\gamma \sim_b \gamma_2$. For $b< \tilde \beta$, since $tord(\gamma_2, X)=\tilde \beta$, we conclude that $\{X,T\}$ is a $b$-cover of $X\cup T$, with $X \cap T$ being $b$-equivalent to $\gamma_1 \cup \gamma_2$. For $\tilde{\beta} \le b$, if $tord(\gamma, \gamma_1) = \beta$, then by Equation \ref{eq:1} we have 
	$\tilde \beta \le b < tord(\gamma, \gamma_2),$ a contradiction with $tord(\gamma, \gamma_2) = \tilde \beta$. Therefore, $tord(\gamma, \gamma_1)>\beta$, which implies $tord(\gamma, \gamma_2)=\beta$ and then $\gamma \sim_{b} \gamma_1$. Thus, $\{X,T\}$ is a $b$-cover of $X\cup T$, with $X \cap T$ being $b$-equivalent to $\gamma_1$.
	
	
	By Theorem \ref{Teo:MayerVietoris} we have the following exact sequence:
	\begin{equation*}
		\cdots \to MDH_{1}^{b}(X\cap T) \to MDH_1^{b}(X) \oplus MDH_1^{b}(T) \to MDH_1^{b}(X\cup T) \to
	\end{equation*}
	\begin{equation*}
		 MDH_{0}^{b}(X \cap T) \to MDH_{0}^{b}(X) \oplus MDH_{0}^{b}(T) \to MDH_0^{b}(X\cup T) \to 0.
	\end{equation*}
	
	By Corollary \ref{Cor: MDH of a NE Holder triangle}, we have $MDH_1^{b}(T) \cong \{0\}$, because $T$ is LNE. Since $X$, $T$ and $X\cup T$ have connected link, $MDH_0^{b}(T) \cong MDH_0^{b}(X) \cong MDH_0^{b}(X\cup T) \cong \mathbb{Z}$. Furthermore $MDH_{1}^{b}(X\cap T) \cong \{0\}$, because $X \cap T$ is $b$-equivalent to a finite union of arcs. Now let $k, l \in \mathbb{Z}_{\ge 0}$ such that $MDH_{1}^{b}(X) \cong \mathbb{Z}^k$ and $MDH_{1}^{b}(X\cup T) \cong \mathbb{Z}^l$. There are two cases:
	
	\textbf{Case 1:} $\beta \le b < \tilde \beta$. Since $X \cap T \sim_b \gamma_1 \cup \gamma_2$, then $MDH_{0}^{b}(X\cap T) \cong \Z^2$ and the exact sequence reduces to
	\begin{equation*}
		\dots \to 0 \stackrel{f_1}{\rightarrow} \mathbb{Z}^k  \stackrel{f_2}{\rightarrow} \mathbb{Z}^l \stackrel{f_3}{\rightarrow} \mathbb{Z}^2  \stackrel{f_4}{\rightarrow} \mathbb{Z}^2 \stackrel{f_5}{\rightarrow} \mathbb{Z} \stackrel{f_6}{\rightarrow} 0.
	\end{equation*}
	By exactness of the sequence we obtain
	\begin{equation*}
		0={\rm dim}({\rm im}(f_1))={\rm dim}({\rm ker}(f_2)) \Rightarrow k={\rm dim}({\rm im}(f_2))= {\rm dim}({\rm ker}(f_3))\Rightarrow 
	\end{equation*}
	\begin{equation*}
		l-k={\rm dim}({\rm im}(f_3))={\rm dim}({\rm ker}(f_4)) \Rightarrow 2-l+k={\rm dim}({\rm im}(f_4))={\rm dim}({\rm ker}(f_5)) \Rightarrow
	\end{equation*}
	\begin{equation*}
		l-k= {\rm dim}({\rm im}(f_5))={\rm dim}({\rm ker}(f_6)).
	\end{equation*}
	Therefore, as ${\rm dim}({\rm ker}(f_6))=1$, we obtain $l=k+1$ and the result follows.
	
	\textbf{Case 2:} $\tilde \beta \le b$. Since $X \cap T \sim_b \gamma_1$, then $MDH_{1}^{b}(X\cap T) \cong \Z$ and the exact sequence reduces to
	\begin{equation*}
	\dots \to 0 \stackrel{f_1}{\rightarrow} \mathbb{Z}^k  \stackrel{f_2}{\rightarrow} \mathbb{Z}^l \stackrel{f_3}{\rightarrow} \mathbb{Z}  \stackrel{f_4}{\rightarrow} \mathbb{Z}^2 \stackrel{f_5}{\rightarrow} \mathbb{Z} \stackrel{f_6}{\rightarrow} 0.
	\end{equation*}
	By exactness of the sequence we obtain
	\begin{equation*}
		0={\rm dim}({\rm im}(f_1))={\rm dim}({\rm ker}(f_2)) \Rightarrow k={\rm dim}({\rm im}(f_2))= {\rm dim}({\rm ker}(f_3))\Rightarrow 
	\end{equation*}
	\begin{equation*}
		l-k={\rm dim}({\rm im}(f_3))={\rm dim}({\rm ker}(f_4)) \Rightarrow 1-l+k={\rm dim}({\rm im}(f_4))={\rm dim}({\rm ker}(f_5)) \Rightarrow
	\end{equation*}
	\begin{equation*}
		1+l-k= {\rm dim}({\rm im}(f_5))={\rm dim}({\rm ker}(f_6)).
	\end{equation*}
	Therefore, as ${\rm dim}({\rm ker}(f_6))=1$, we obtain $l=k$ and the result follows.
\end{proof}

\begin{remark}\label{Rem: MD-Homology for b less than mu(X)}
	Notice that, for any surface $X$, when $b < \mu(X)$, every two arcs in $V(X)$ are $b$-equivalent (see Remark \ref{Rem: b-equiv and tord of arcs}). Therefore, its MD-Homology of degree $1$ in this case is the same as of an arc, i.e., $\{0\}$. Therefore, we will only deal with the case $b\ge \mu(X)$.
\end{remark}

\begin{Exam}\label{Exam: bubble snake}
	A $\beta$-bubble $X=T(\gamma_{1},\gamma_{2})$ is a non-singular $\beta$-H\"older triangle containing an interior arc $\theta$ such that both $X_1 = T(\gamma_{1},\theta)$ and $X_2 = T(\theta,\gamma_{2})$ are LNE and $tord(\gamma_{1},\gamma_{2}) > \beta$. When $X$ is a snake it is called a $\beta$-bubble snake (see Definition 4.45 of \cite{GabrielovSouza}). 
	
	The Definition of a $\beta$-bubble allows a great variety of bubbles, however, as showed in Theorem 5.9 of \cite{GabrielovSouza}, the important ones which emerge when decomposing the Valette link of a surface into normal and abnormal zones are outer Lipschitz equivalent to the ones with link as in Figure 10 of \cite{GabrielovSouza}.
	
	For a concrete example of a $\beta$-bubble snake consider the set $X = T_1\cup T_2$ with $T_1 = \Delta(\gamma_{1},\lambda_1)$ and $T_2 = \Delta(\lambda_1,\lambda_2)\cup \Delta(\lambda_2,\gamma_{2})$ where, for $\alpha,\beta \in \F$, $\alpha > \beta$, one has $\gamma_i(t)=(t,(-1)^{i}t^\alpha,0)$ and $\lambda_i(t) = (t,(-1)^{i}t^\beta,t^\beta)$, for $i=1,2$, are arcs parameterized by the first coordinate, which is equivalent to the distance to the origin. The H\"older triangles $T_1$, $T_2$, $T(\lambda_1,\lambda_2)$ and $T(\lambda_2,\gamma_{2})$ are LNE. In particular, $X$ is non-singular. Notice that $tord(\gamma_1,\gamma_2) = \alpha > \beta = itord(\gamma_1,\gamma_2)$, since $d_{inn}(\gamma_1(t),\gamma_2(t)) \ge 2t^\beta$. Therefore, $X$ is a $\beta$-bubble snake. 	
	
	It easily follows from Remark \ref{Rem: b-equiv and tord of arcs}, by considering $X=T_1$ and $T=T_2$ in Lemma \ref{Lem: gluing-triangles}, that the MD-Homology group of a $\beta$-bubble snake $X$ is $$MDH_{1}^{b}(X,0,d_{out}) = \left\{\begin{array}{cc}
		0, & \text{for } b < \beta\\
		\mathbb{Z}, & \text{for } \beta\leq b <  \alpha\\
		0, & \text{for } \alpha\leq b \\
	\end{array}\right.$$
	However, the MD-Homology group of a non-snake bubble could be wild (see Remark \ref{Rem: non-snake realization}.)
\end{Exam}

\begin{remark}
	If $X$ is a $\beta$-horn, which is any set germ in $\R^n$ that is outer Lipschitz equivalent to $\{(x,y,z)\in \R^3\mid x^{\beta} + y^\beta = z^{2\beta}\}$, it is $b$-equivalent to a $\beta$-bubble snake of spectra $\{\alpha\}$ when $\beta \le b < \alpha$. In particular, $$MDH_{1}^{b}(X,0,d_{out}) = \left\{\begin{array}{cc}
		0, & \text{for } b < \beta\\
		\mathbb{Z}, & \text{for } \beta\leq b  \\
	\end{array}\right.$$
\end{remark}

\begin{remark}\label{Rem: gluing bubbles}
	Consider the word $W = [w_1\cdots w_p]$, in the alphabet $\{x_1,x_2,\ldots \}$, where $p = 2k$, $k \in \mathbb{Z}_{\ge 2}$, obtained recursively as follows: if $k=2$ then $W = W_2 = [x_1x_2x_1x_2]$ and for $k > 2$ we have $W=W_k = [w_1\cdots w_{r-1}x_pw_r x_p]$, where $W_{k-1} = [w_1\cdots w_r]$ and $w_r = w_{r-2} = x_{k-1}$. One can easily check that this word $W$ is a snake name (see Definition 6.6 of \cite{GabrielovSouza}). Hence, by Theorem 6.23 of \cite{GabrielovSouza}, for any $\alpha,\beta \in \F$, with $\alpha>\beta$, there is a $\beta$-snake $X' \subset \R^{2p - 1}$ such that $W(X') = W_k$ (see Example \ref{Exam: realization for snake names}). Notice that $X'$ has exactly $k$ nodes, each with two nodal zones and spectrum $\{\alpha\}$ (see Figure \ref{Fig. snake realization construction}).  
	
	\begin{figure}[h!]
		\centering
		\includegraphics[width=3in]{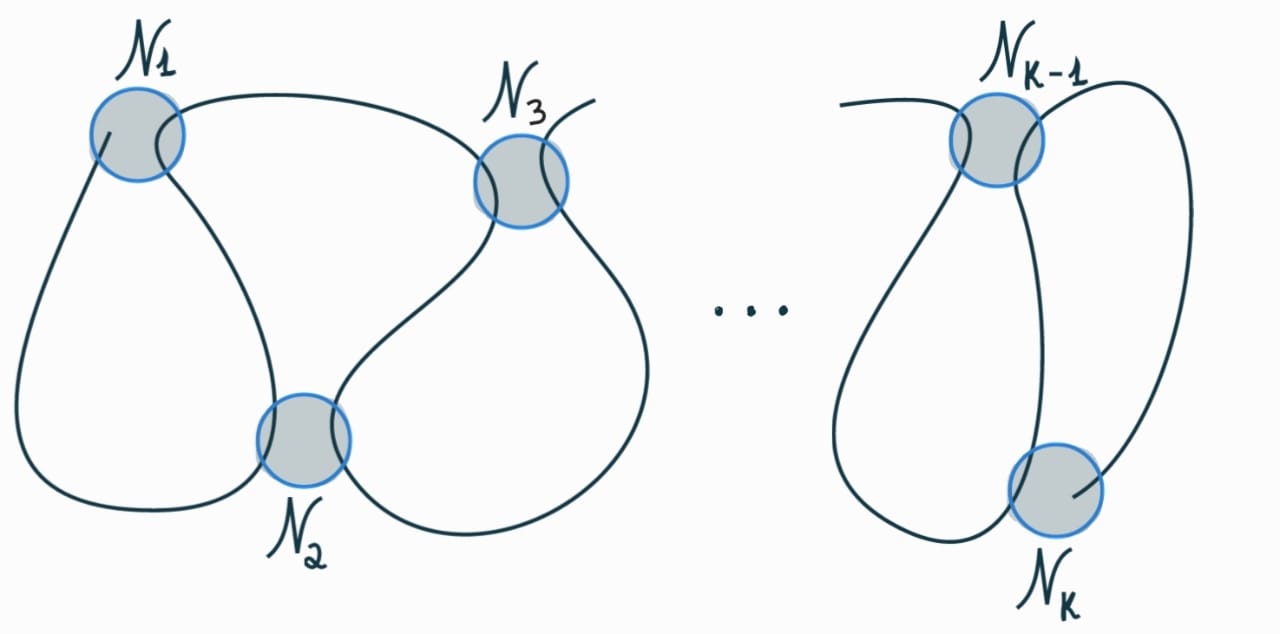}
		\caption{The link of a $\beta$-snake with $k$ nodes $\mathcal{N}_1,\ldots,\mathcal{N}_k$ as in Remark \ref{Rem: gluing bubbles}. Each $\delta_j$ belong to a node and $\sigma_j \in G(T(\delta_j,\delta_{j+1}))$, $1\le j < 2p$. Points inside the shaded disks represent arcs with tangency order higher than $\beta$.}\label{Fig. snake realization construction}
	\end{figure}

	Slightly changing the choices of arcs in the construction of $X'$ we can obtain a snake $X \subset \R^{2p - 1}$ with the same number of nodes, but with a different spectra. Indeed, if $k = \tilde{k}_1 + \cdots +\tilde{k}_m$, $\tilde k_i \in \mathbb{Z}_{\ge 1}$, then we could have $\tilde{k}_i$ of the nodes with spectrum $\{q_i\}$, $\beta <q_i$, for each $i=1,\ldots, m$. In fact, let $J = \{1\le j \le p \mid w_i = w_j \; \textrm{for} \; i<j\}$. By the construction of $W$ we know that we have exactly $k$ distinct letters in $W$, each one appearing twice. Assume that $J = \bigsqcup_{l=1}^m J_l$ where each $J_l$ has exactly $\tilde{k}_l$ elements for $1\le l \le m$. The changing consists of setting, for each $j\in J_l$, set $\delta_j = \delta_{r(j)} + t^{q_{l}}e_j$ (see Example \ref{Exam: realization for snake names}). Since we have $q_l>\beta$ for each $l$, results from 6.18 to 6.23 of \cite{GabrielovSouza} still hold. Therefore, $X$ is a $\beta$-snake such that $W(X) = W$. 
\end{remark}

\begin{Lem}\label{Lem: MD-Homology of glued triangles}
	Given $\alpha,q_1,\ldots,q_m,\beta \in \F$, with $\alpha>\beta$ and $q_l>\beta$ for all $l$, let $X'$ and $X$ be $\beta$-snakes as in Remark \ref{Rem: gluing bubbles}. Then
	\begin{enumerate}
		\item $$MDH_{1}^{b}(X',0,d_{out}) = \left\{\begin{array}{cc}
			0, & \text{for } b < \beta\\
			\mathbb{Z}^{k}, & \text{for } \beta\leq b <  \alpha\\
			0, & \text{for } \alpha\leq b \\
		\end{array}\right.$$
	\item If $q_1<q_2<\cdots < q_m$, then $$MDH_1^b(X,0,d_{out}) = 
	\left\{\begin{array}{cc}
		0, & \text{for } b < \beta\\
		\mathbb{Z}^{k}, & \text{for } \beta\leq b <  q_1\\
		\mathbb{Z}^{k- \tilde{k}_1}, & \text{for } q_1\leq b <  q_2\\
		\vdots &  \vdots\leq b <  \vdots\\
		\mathbb{Z}^{k - (\tilde{k}_1+\cdots+\tilde{k}_{m-1})}, & \text{for } q_{m-1}\leq b <  q_{m}\\
		0, & \text{for } q_{m}\leq b \\
	\end{array}\right.$$
	\end{enumerate}
\end{Lem}
\begin{proof}
	$(1)$ When $b<\beta$ the result follows from Remark \ref{Rem: MD-Homology for b less than mu(X)}. If $b>\alpha$ then any two arcs in $X'$ which are $b$-equivalent are necessarily in the same connected component of a $\mathcal{H}_{\alpha,\eta}(X')$, for $\eta > 0$ sufficiently small. Since $\mathcal{H}_{\alpha,\eta}(X')$ admits any arc in $V(X')$ as $b$-deformation retract, $b > \alpha$ implies that $X'$ is $b$-contractible, and then its MD-Homology is $\{0\}$. Consequently, it only remains to prove the case $\beta \le b <\alpha$. In this case, $X'$ is $b$-equivalent to the union of $k$ $\beta$-horns such that each pair of such horns either do not intersect or intersect along a LNE $\beta$-H\"older triangle (one can visualize the link of this surface considering a link as in Figure \ref{Fig. snake realization construction} where arcs in a same node collapsed to a single arc). Let us prove that each of these horns correspond to a generator, hence its MD-Homology group will be $\Z^k$.
	
	Let $\{X_j = T(\gamma_{j-1},\gamma_j)\}_{j=1}^{p-1}$ be a pancake decomposition as in Proposition 4.56 of \cite{GabrielovSouza}, i.e., we choose $\gamma_0=\delta_1$, $\gamma_{p-1} = \delta_{p}$ and each $\gamma_j$, $0<j<p-1$, to be in the $j$-th interior nodal zone of $X'$ when we are moving along its link with orientation from $\gamma_0$ to $\gamma_{p-1}$. Assume that $\gamma_0 \in \mathcal{N}_1$ and $\gamma_{2k} \in \mathcal{N}_k$, where $\mathcal{N}_1$ and $\mathcal{N}_{k}$ are assumed to be the first and last node according to the orientation from $\gamma_0$ to $\gamma_{p-1}$ (see Figure \ref{Fig. snake realization construction}). By the construction of $X'$, for each $j=1,\ldots,p-3$, either $X_j\cup X_{j+1}$ or $X_j\cup X_{j+1}\cup X_{j+2}$ is $b$ equivalent to a $\beta$-bubble. The result then follows from recursively applying Lemma \ref{Lem: gluing-triangles} when adding the pancake $X_j$, for $j>2$, to the surface $\bigcup_{i<j}X_i$.   \\
	
	$(2)$ By the same reasons as in item $(1)$ we have the MD-Homology group of $X$ equal to $\{0\}$ when either $b <\beta$ or $b\ge q_m$, since $q_m > \max_{i}\{q_i\}$. Moreover, $q_1<q_2<\cdots < q_m$ implies that, for $\beta \le b <q_1$, $MDH_1^b(X) \cong MDH_1^b(X') \cong \Z^k.$
	When $q_i \le b < q_{i+1} $, all the arcs with tangency order higher than or equal to $q_{i+1}$ are $b$-equivalent, while the arcs with tangency order at most $q_i$ are not (see Remark \ref{Rem: b-equiv and tord of arcs}). Therefore, in this case, there is $l \in \mathbb{Z}_{\ge 1}$ such that $X$ is $b$-equivalent to the LNE set $\bigcup_{i=1}^{l} \tilde X_i$, where $\tilde X_i\cap \tilde X_j$ is $\{0\}$, an arc or a $\beta$-H\"older triangle, for $1\le i<j \le l$, and where $k-(\tilde k_1+\ldots+\tilde k_i)$ of the $\tilde X_i$'s are $\beta$-horns and the remaining ones are $\beta$-H\"older triangles (see Figure \ref{Fig. gluing bubbles}).

	\begin{figure}[h!]
		\centering
		\includegraphics[width=4in]{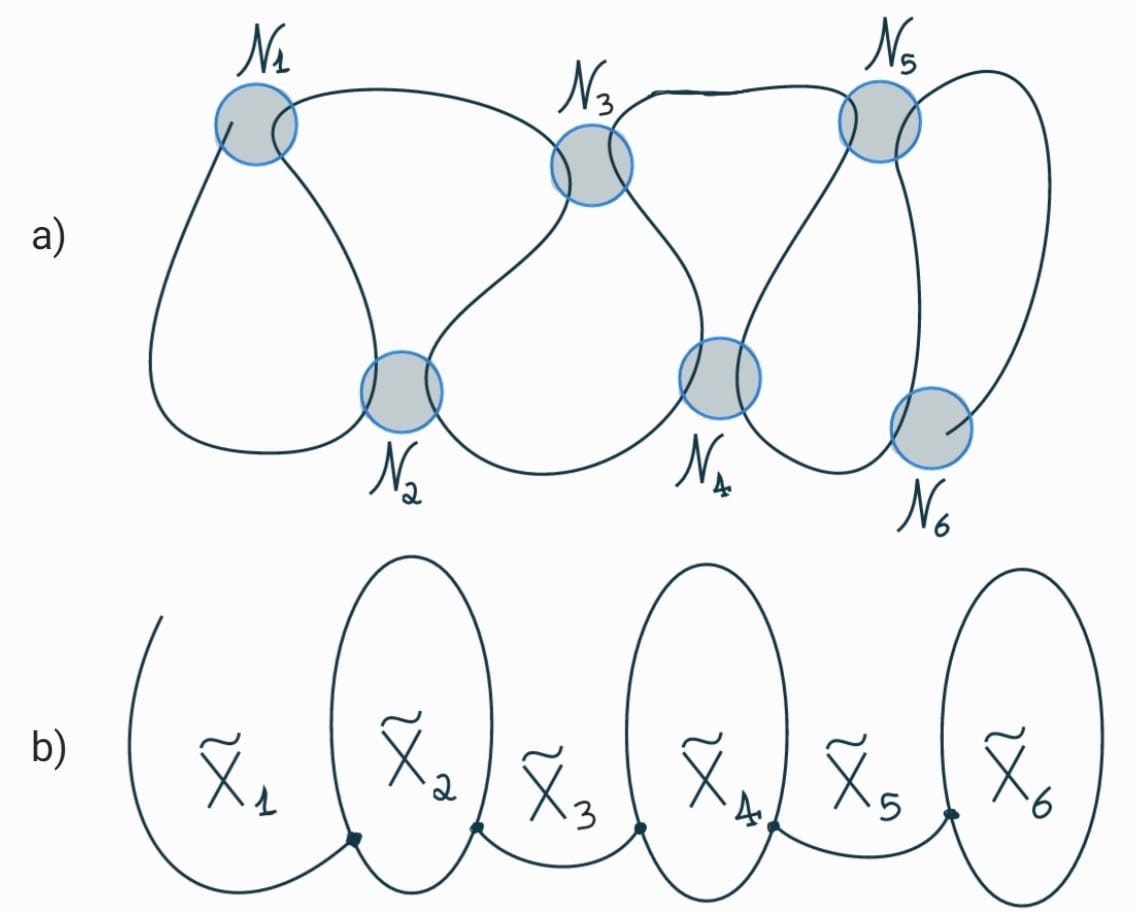}
		\caption{Illustration for the proof of item $(2)$ of Lemma \ref{Lem: MD-Homology of glued triangles}. For the surface which link is part $a)$, the nodes $\mathcal{N}_1, \mathcal{N}_5$ have spectrum $\{q_1\}$, while $\mathcal{N}_3$ has spectrum $\{q_2\}$ and $\mathcal{N}_2, \mathcal{N}_4, \mathcal{N}_6$ have spectrum $\{q_3\}$. The link in part $b)$ corresponds to a surface which is $b$-equivalent to the surface in part $a)$ for $q_1<q_2\le b < q_3$. In part $b)$, $\tilde X_i$, $i=1,3,5$, is the link of a $\beta$-H\"older triangle and $\tilde X_j$, $j=2,4,6$, is the link of a $\beta$-horn. Points inside the shaded disks represent arcs with tangency order higher than the surface's exponent.}\label{Fig. gluing bubbles}
	\end{figure}
	
	Since each $\tilde X_i$ that is a $\beta$-H\"older triangle is $b$-contractible and each $\tilde X_i$ that is a $\beta$-horn increase the number of generators by one (by the same recursive application of Lemma \ref{Lem: gluing-triangles} of the previous case), we conclude that $$MDH_1^b(X,0,d_{out}) \cong \Z^{k-(\tilde k_1+\cdots+\tilde k_i)}.$$
\end{proof}

\begin{Teo}\label{Teo: realization for MD-Homology}
	For any sequences of non-negative integers, $k_1 < k_2 < \cdots < k_m$, and rationals $q_1<q_2< \cdots < q_{m}$, with $q_1>\beta=\mu(X) \ge 1$, there exists a surface $X$ such that $$MDH_1^b(X,0,d_{out}) =
	\left\{\begin{array}{cc}
		0, & \text{for } b < \beta\\
		\mathbb{Z}^{k_m}, & \text{for } \beta\leq b <  q_1\\
		\mathbb{Z}^{k_{m-1}}, & \text{for } q_1\leq b <  q_2\\
		\vdots &  \vdots\leq b <  \vdots\\
		\mathbb{Z}^{k_1}, & \text{for } q_{m-1}\leq b <  q_{m}\\
		0, & \text{for } q_{m}\leq b \\
	\end{array}\right.$$
\end{Teo}
\begin{proof}
	If $m=1$ and $k_1=1$ then the sequences $k_1 < k_2 < \cdots < k_m$ and $q_1<q_2< \cdots < q_{m}$ reduce to $1$ and $q_1=q$. In this case, the result holds as seen in Example \ref{Exam: bubble snake} taking $\mu(X) = \beta$ and $q = \alpha$. Otherwise, the result follows from Lemma \ref{Lem: MD-Homology of glued triangles}, item $(2)$, where we take $k=k_m$, $\tilde k_{l} = k_{m-l+1} - k_{m-l}$, for $l=1,\ldots, m-1$, and $\tilde k_m = k_1$.

\end{proof}

\begin{remark}\label{Rem: non-snake realization}
	Notice that under the hypothesis of Theorem \ref{Teo: realization for MD-Homology} one can obtain a non-snake $\beta$-bubble with the same MD-Homology of $X$. Indeed, given $m=2k+1$, where $k\ge2$, we consider the arcs $\delta_1, \ldots,\delta_{m}$ and $\sigma_1,\ldots,\sigma_{m-1}$ in $\R^{2m - 1} = \langle e_1,\ldots,e_{2m-1}\rangle$ defined by $\delta_1(t) = te_1$, $\delta_j(t) = \delta_1(t) + t^{\beta} e_{j}$ for $1<j\le k+1$, $\delta_j(t) =\delta_{2k+2-j}(t) + t^{\alpha_{j-k-1}}e_j$ for $k+1 < j \le m$ and $\sigma_i(t) = \delta_1(t) + t^\beta e_{m+i}$ for $1\le i <m$. If $1\le \beta <\alpha_1<\cdots < \alpha_k$ then Corollary 6.20 and Lemma 6.21 of \cite{GabrielovSouza} implies that $X_1 = \bigcup_{1\le j\le k}T_j$ and $X_2 = \bigcup_{k<j\le m}T_j$ are two LNE non-singular $\beta$-H\"older triangles. It also follows from those results that $\delta_{k+1}$ is Lipschitz non-singular. Hence, since \begin{enumerate}
		\item for all $1\le i, j\le m$ and $i\ne j$, $tord(\delta_i,\delta_j) = \left\{\begin{array}{cc}
			\alpha_i, & \text{if } j = 2k + 2 - i\\
			\beta, & \text{otherwise } 
		\end{array}\right.$,
	\item $tord(\delta_i,\sigma_j) = \beta$ for all $1\le i\le m$ and $1\le j\le m-1$,
	\item $tord(\sigma_i,\sigma_j) = \beta$ for all $1\le i,j \le m-1$ with $i\ne j$,
	\end{enumerate}
	we have that $X = \bigcup_{1\le j\le m}T_j$ is a non-snake $\beta$-bubble with $\theta = \delta_{k+1}$ (see Example \ref{Exam: bubble snake}). Nevertheless, for $\alpha_i \le b < \alpha_{i+1}$, $i=0, 1,\ldots, k-1$, with $\alpha_0 = \beta$ by convention, $X$ has the same MD-Homology as a set which is the union of $k - i$ $\beta$-horns such that each pair of such horns either do not intersect or intersect along a LNE $\beta$-H\"older triangle. Hence, its MD-Homology group is $$MDH_1^b(X,0,d_{out},\Z)=
	\left\{\begin{array}{cc}
		0, & \text{for } b < \beta\\
		\mathbb{Z}^{k}, & \text{for } \beta\leq b <  \alpha_1\\
		\mathbb{Z}^{k-1}, & \text{for } \alpha_1\leq b <  \alpha_2\\
		\vdots &  \vdots\leq b <  \vdots\\
		\mathbb{Z}, & \text{for } \alpha_{m-1}\leq b <  \alpha_{m}\\
		0, & \text{for } \alpha_{m}\leq b \\
	\end{array}\right.$$
	Slightly adapting this construction, by considering some of the $\alpha_i$ equal, one can obtain the MD-Homology as in Theorem \ref{Teo: realization for MD-Homology}.
	We illustrate this construction on Figure \ref{Fig. non-snake bubble}.
	
	\begin{figure}[h!]
		\centering
		\includegraphics[width=3in]{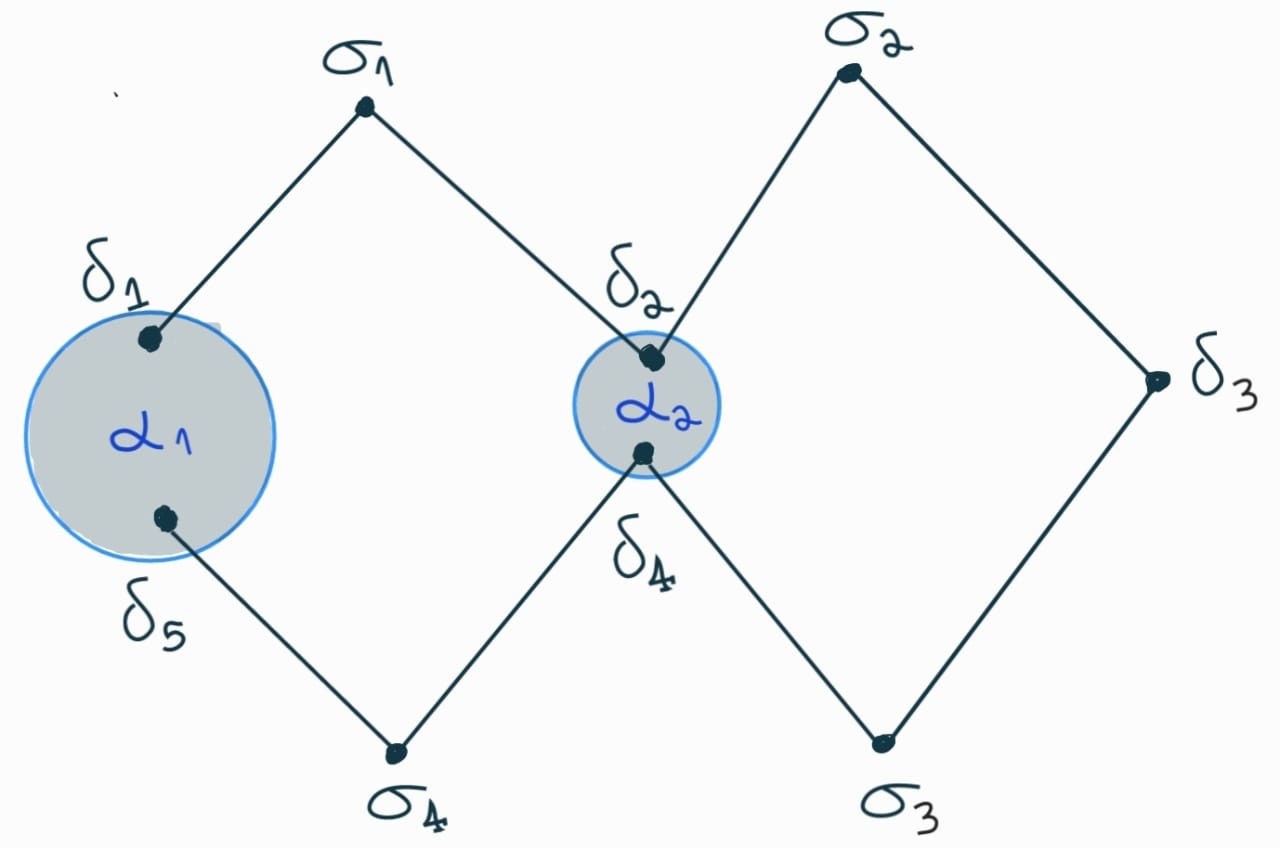}
		\caption{The link of a non-snake $\beta$-bubble in the construction of Remark \ref{Rem: non-snake realization} for $k=2$. Points inside the shaded disks represent arcs with tangency order higher than the surface's exponent.}\label{Fig. non-snake bubble}
	\end{figure}
\end{remark}

\section{Surfaces with the same MD-Homology}\label{sec:weakouter-MDH}
It is known from the work in \cite{MDH2022} that outer Lipschitz equivalent surfaces have the same MD-Homology in the outer metric. A natural question one could formulate is: is there a weaker condition on the surface guaranteeing they have the same MD-Homology? We have a positive answer in the case where the surfaces are simples $\beta$-snakes.

Gabrielov and Souza defined in \cite{GabrielovSouza} the notion of weakly outer Lipschitz equivalence for $\beta$-H\"older triangles, which is a homeomorphism $h\colon X\rightarrow Y$ such that $tord(\gamma,\gamma')>\beta \Leftrightarrow tord(h(\gamma),h(\gamma'))>\beta$ for any arcs $\gamma,\gamma'\subset X$. In Theorem 6.28 of \cite{GabrielovSouza} they establish necessary and sufficient conditions so that two $\beta$-snakes are weakly outer equivalent. 

One could be initially tempted to imagine that those notions, ``having the same MD-Homology'' and ``being weakly outer Lipschitz equivalent'' could imply one another. However this does not happen as it is shown in Example \ref{Exam: weakly outer equivalent surfaces}.

\begin{Exam}\label{Exam: weakly outer equivalent surfaces}
	Let $X$ and $\tilde{X}$ be $\beta$-snakes with link as in Figure \ref{Fig. snakes weakly outer equiv}, from the left to the right, respectively. The Valette links of those surfaces contain two segments with multiplicity $2$: the ones with adjacent nodal zones in the nodes $\mathcal{N}, \mathcal{N}'$ (say $S, S'$) and $\mathcal{\tilde{N}}, \mathcal{\tilde{N}}'$ (say $\tilde{S}, \tilde{S}'$). Let us further assume that $1\le\beta < \alpha_1 <\alpha_2$ are exponents such that $$\inf\{tord(\gamma,\gamma') \mid \gamma \in S,\gamma'\in S'\} = \alpha_1 = \inf\{tord(\gamma,\gamma') \mid \gamma \in \tilde{S},\gamma'\in \tilde{S}'\},$$ 
	$$\sup\{tord(\gamma,\gamma') \mid \gamma \in S,\gamma'\in S'\} = \alpha_2 = \sup\{tord(\gamma,\gamma') \mid \gamma \in \tilde{S},\gamma'\in \tilde{S}'\}.$$ 
	Since both the tangency orders between nodal zones in a same node and segments with multiplicity $2$ are higher than $\beta$, $X$ and $\tilde{X}$ are weakly outer bi-Lipschitz equivalent by Theorem 6.28 of \cite{GabrielovSouza}. However, even in this case, they are not necessarily outer Lipschitz equivalent, because we could have distinct spectra or, less obvious, that $\beta'$ and $\tilde{\beta}'$, with $\beta < \beta' < \tilde{\beta}'$, are the supremum of the orders of zones $Z\subset S$ and $\tilde{Z} \subset \tilde{S}$, respectively, such that $tord(Z,S') = \alpha_2 = tord(\tilde{Z},\tilde{S}')$. Nevertheless, assuming that the spectra of the nodes is $\{\alpha_1\}$ for both surfaces, it follows from Lemma \ref{Lem: MD-Homology of glued triangles} that those surfaces have the same MD-Homology, given as 
	$$MDH_{1}^{b}(X,0,d_{out}) = MDH_{1}^{b}(\tilde X,0,d_{out})= \left\{\begin{array}{cc}
		0, & \text{for } b < \beta\\
		\mathbb{Z}^{3}, & \text{for } \beta\leq b <  \alpha_1\\
		\mathbb{Z}, & \text{for } \alpha_1\leq b <  \alpha_2\\
		0, & \text{for } \alpha_{2}\leq b \\
	\end{array}\right.$$
	
	Finally, notice that a $\beta$-snake with link as in Figure \ref{Fig. snake word}, with spectrum of $x_1$ and $x_2$ equal to $\{\alpha_1\}$ and spectrum of $x_3$ equal to $\{\alpha_2\}$, has the same MD-Homology as $X$ and $\tilde{X}$, but it is not weakly outer Lipschitz equivalent to them, since all of the segments in its Valette link have multiplicity $1$.
\end{Exam}

\begin{figure}[h!]
	\centering
	\includegraphics[width=4in]{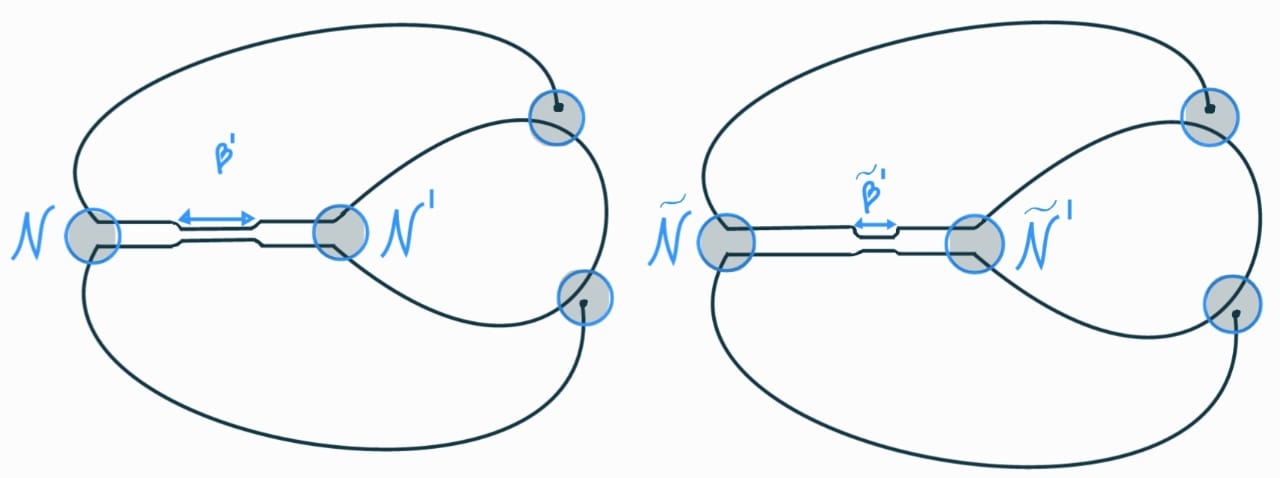}
	\caption{The link of two snakes with the same exponent, which are weakly outer Lipschitz equivalent but not outer Lipschitz equivalent. Points inside the shaded disks represent arcs with tangency order higher than the surface's exponent.}\label{Fig. snakes weakly outer equiv}
\end{figure}

Below we will generalize the definition of multiplicity of an arc in a snake present in \cite{GabrielovSouza}.

\begin{Def}\label{Def: alpha-multiplicity}
	Let $X$ be a H\"older triangle and $\gamma \subset X$ an arc. For $\alpha \in \mathbb{F}$, $\alpha \ge \mu(X) \ge 1$, the \textbf{$\alpha$-multiplicity of $\gamma$}, denoted by $m_{X,\alpha}(\gamma)$, is defined as the number of connected components of $\mathcal{H}_{\alpha,\eta}(\gamma) \cap X$ for $\eta>0$ small enough (see Remark 4.14 in \cite{GabrielovSouza}). In particular, when $Z\subset V(X)$ is a zone and $\gamma \notin Z$, we define $m_{X,\alpha,Z}(\gamma)=2$ if $(\mathcal{H}_{\alpha,\eta}(\gamma) \cap X)\cap T \neq \emptyset$, for some H\"older triangle $T$ with $V(T) \subset Z$. Otherwise we define $m_{X,\alpha,Z}(\gamma)=1$. Moreover, we define $m_{X,\alpha}(Z,Z')=2$, for some zone $Z'\subset V(X)$, if $m_{X,\alpha,Z'}(\gamma)=2$ for some arc in $\gamma \in Z$ and we define $m_{X,\alpha}(Z,Z')=1$ otherwise.  
\end{Def}

\begin{remark}
	When $X$ is a snake and $\alpha = \mu(X)$, the $\alpha$-multiplicity of $\gamma$ as in Definition \ref{Def: alpha-multiplicity} coincides with the multiplicity of $\gamma$ as in Definition 4.13 in \cite{GabrielovSouza}. 
\end{remark}

\begin{Def}\label{Def:subsegment}
	Let $X$ be a $\beta$-snake and $Z,Z'\subset V(X)$ be disjoint zones such that $tord(Z,Z')>\beta$. Let $\alpha \in \F$, $\alpha > \beta$, be such that 
	$$E_{Z,Z'}^\alpha := \{\gamma' \in Z \mid tord(\gamma',Z') = \alpha\} \neq \emptyset.$$ 
	
	We call the set $E_{Z,Z'}^\alpha$ an \textbf{$\alpha$-subsegment} (resp., \textbf{$\alpha$-subnodal zone}) of $Z$ if $Z,Z'$ is a pair of segments (resp., nodal zones).
\end{Def}

The analogs of some results in this Proposition bellow were dealt, with a different notation, for pairs of disjoint normally H\"older triangles, in Section 4 of \cite{Pairsoftriangles}.
\begin{Prop}
	Let $X$ be a $\beta$-snake and $Z,Z'\subset V(X)$ disjoint zones. If $Z$ and $Z'$ are both segments or nodal zones of $X$, then
	\begin{enumerate}
		\item We have $\beta \le tord(Z,Z') <\infty$. In particular, for any arc $\gamma \in Z$ there exists $\alpha\in \F$ such that $m_{X,\alpha}(\gamma) = 1$.
		\item Let $\alpha = tord(\gamma,Z')$, for some $\gamma \in Z$. Define
		$$H_{\gamma} := \{\gamma' \in V(X) \mid itord(\gamma,\gamma') > \alpha\}.$$
		Then $tord(H_\gamma,Z') = \alpha$ and $H_\gamma \subset Z$. In particular, if the set $E_{Z,Z'}^\alpha$ is not empty, for some $\alpha \ge \beta$, then $E_{Z,Z'}^\alpha$ is a union of zones and $\mu(E_{Z,Z'}^\alpha)\le \alpha$.
		\item For any $\alpha>\beta$, we have $\mu(E_{Z,Z'}^\alpha) = \mu(E_{Z',Z}^\alpha)$. 
		\item If, for $\alpha>\beta$, $E_{Z,Z'}^\alpha \neq \emptyset$, then it is the union of finitely many disjoint zones. Consequently, for each $\alpha > \beta$, there are finitely many $\alpha$-subsegments and $\alpha$-subnodal zones.
	\end{enumerate}
	
\end{Prop}
\begin{proof}
	(1) Since $Z,Z' \subset V(X)$ we necessarily have $\beta \le tord(Z,Z')$. Moreover, in the case $Z$ and $Z'$ are either segments in different clusters (see Definition 4.32 in \cite{GabrielovSouza}) or nodal zones in different nodes we would have $tord(Z,Z')=\beta$. The remaining inequality follows from Propositions 4.20 and 4.27 in \cite{GabrielovSouza}, since $Z\cap Z' = \emptyset$.\\
	
	(2) and (3) are consequences of Lemmas 2.14, 4.24 and Proposition 4.27 of \cite{GabrielovSouza}. Indeed, they show that for a segment (or nodal) arc $\gamma$, we have that $itord(\gamma,\gamma') >\beta$ implies $m_{X,\alpha}(\gamma) = m_{X,\alpha}(\gamma') $. In particular, $itord(\gamma,\gamma') >\beta$ implies that $\gamma$ and $\gamma'$ belong to the same segment (or nodal zone). Therefore, $H_\gamma \subset Z$ and $\mu(E_{Z,Z'}^\alpha) \le \alpha$.\\
	
	(4) Suppose that, for a given $\alpha>\beta$, we have $E_{Z,Z'}^\alpha = \bigsqcup_{k=1}^\infty Z_k$. For any two $Z_j,Z_k$, $j \neq k$, and any arcs $\gamma_j \in Z_j$ and $\gamma_k \in Z_k$, there should exist $\gamma_{j,k} \in V(T(\gamma_j,\gamma_k))\setminus(Z_j\cup Z_k)$, otherwise, by the definition of $E_{Z,Z'}^\alpha$, we would have $Z_j\cap Z_k \neq \emptyset$, a contradiction. Thus, for any pair $j,k$ there must exist some $\alpha_{j,k}\neq \alpha$ where $\gamma_{j,k} \in E_{Z,Z'}^{\alpha_{j,k}}$. Let us assume, without loss of generality, that for $k\ge 2$ we have $Z_k \subset V(T(\gamma_{k-1},\gamma_{k+1}))$, where $\gamma_j \in Z_j$, $j=k-1,k+1$. In particular, for each pair $j,k$ there will exist some $\alpha_{j,k}\neq \alpha$ such that $E_{Z,Z'}^{\alpha_{j,k}} \subset V(T(\gamma_{k-1},\gamma_{k+1}))$, where $\gamma_j \in Z_j$, $j=k-1,k+1$. However, this implies that either we will have a certain value of $b$ for which we count infinitely many generators of the MD-homology of degree 1, a contradiction with Theorem \ref{Teo:finitude-geradores}, or we will have infinitely many jumping rates of $b$, a contradiction with Theorem \ref{Teo:finitude-jumping-rate}.

\end{proof}

\begin{Def}\label{Def:Simple-basic-snake}
	Let $X$ be a $\beta$-snake and $Z,Z'\subset V(X)$ disjoint zones which are both segments or nodal zones of $X$ with $tord(Z,Z') > \beta$. Let $\alpha_{Z,Z'} = tord(Z,Z')$. We say that $Z$ and $Z'$ have \textbf{simple Lipschitz contact} if, for every $\alpha_1,\alpha_3\in \F$, with $\alpha_1\le \alpha_3\le \alpha_{Z,Z'}$, such that there is $\alpha_{2}\in \F$ satisfying $E_{Z,Z'}^{\alpha_i}\neq \emptyset$, $i=1,2,3$, and $E_{Z,Z'}^{\alpha_2} \subset V(T(\gamma_1,\gamma_3))$, where $\gamma_j \in E_{Z,Z'}^{\alpha_j}$, $j=1,3$, we have $\alpha_1\le \alpha_{2}\le \alpha_3$. We say that $X$ is a \textbf{simple $\beta$-snake} if any of its pairs of distinct segments or nodal zones with tangency order higher than $\beta$ has simple Lipschitz contact. A \textbf{basic $\beta$-snake} is a simple $\beta$-snake where its spectra is a singleton and all of its segments have multiplicity $1$.
\end{Def}

\begin{remark}
	Notice that a $\beta$-snake obtained from the construction of Theorem 6.23 of \cite{GabrielovSouza} is a basic $\beta$-snake.
\end{remark}

\begin{Teo}\label{Teo: HMD of basic snakes}
	Let $X$ be a basic $\beta$-snake with spectra $\{\alpha\}$. If $X$ contains exactly $k$ nodes and $m$ nodal zones, then $$MDH_{1}^{b}(X,0,d_{out},\Z) = \left\{\begin{array}{cc}
		0, & \text{for } b < \beta\\
		\mathbb{Z}^{m - k}, & \text{for } \beta\leq b <  \alpha\\
		0, & \text{for } \alpha\leq b \\
	\end{array}\right.$$ 
\end{Teo}
\begin{proof}
	For $\beta \le b < \alpha$, Remark \ref{Rem: MD-Homology for b less than mu(X)} implies that $X$ is $b$-equivalent to a surface $\tilde X$ whose canonical H\"older complex, say $(G,\sigma)$, satisfies that $G$ is a connected graph with exactly $k$ vertices (the nodes of $X$), $m-1$ edges (the segments of $X$) and $\sigma \equiv \beta$ (see Proposition 4.30 of \cite{GabrielovSouza}). Hence, by Theorem \ref{Teo: formula for inner HMD}, $$MDH_{1}^{b}(X,0,d_{out},\Z) = MDH_{1}^{b}(\tilde X,0,d_{out},\Z) = \left\{\begin{array}{cc}
		0, & \text{for } b < \beta\\
		\mathbb{Z}^{(m-1) - k + 1}, & \text{for } \beta\leq b <  \alpha\\
		0, & \text{for } \alpha\leq b \\
	\end{array}\right.$$ 
\end{proof}

\begin{Cor}\label{Cor: HMD through snake names}
	Let $X$ be a basic $\beta$-snake with spectra $\{\alpha\}$. Let $W=W(X)$ be the snake name associated with $X$. If $W$ has length $m$ and exactly $k$ distinct letters, then $$MDH_{1}^{b}(X,0,d_{out},\Z) = \left\{\begin{array}{cc}
		0, & \text{for } b < \beta\\
		\mathbb{Z}^{m - k}, & \text{for } \beta\leq b <  \alpha\\
		0, & \text{for } \alpha\leq b \\
	\end{array}\right.$$
\end{Cor}
\begin{proof}
	Since $W$ has length $m$ and exactly $k$ distinct letters, $X$ has exactly $m$ nodal zones and $k$ nodes. Therefore, the result follows from Theorem \ref{Teo: HMD of basic snakes}.
\end{proof}

\begin{remark}
	In the hypothesis of Corollary \ref{Cor: HMD through snake names}, the number of generators of the MD-Homology coincides with the number of distinct semi-primitive subwords of $W$ (see Definition 6.5 of \cite{GabrielovSouza}). Moreover, there are basic snakes with the same MD-Homology which are not weakly outer Lipschitz equivalent, since they have different number of nodes (see Figure \ref{Fig. snakes not weakly outer equiv}).
\end{remark}

\begin{figure}[h!]
	\centering
	\includegraphics[width=3in]{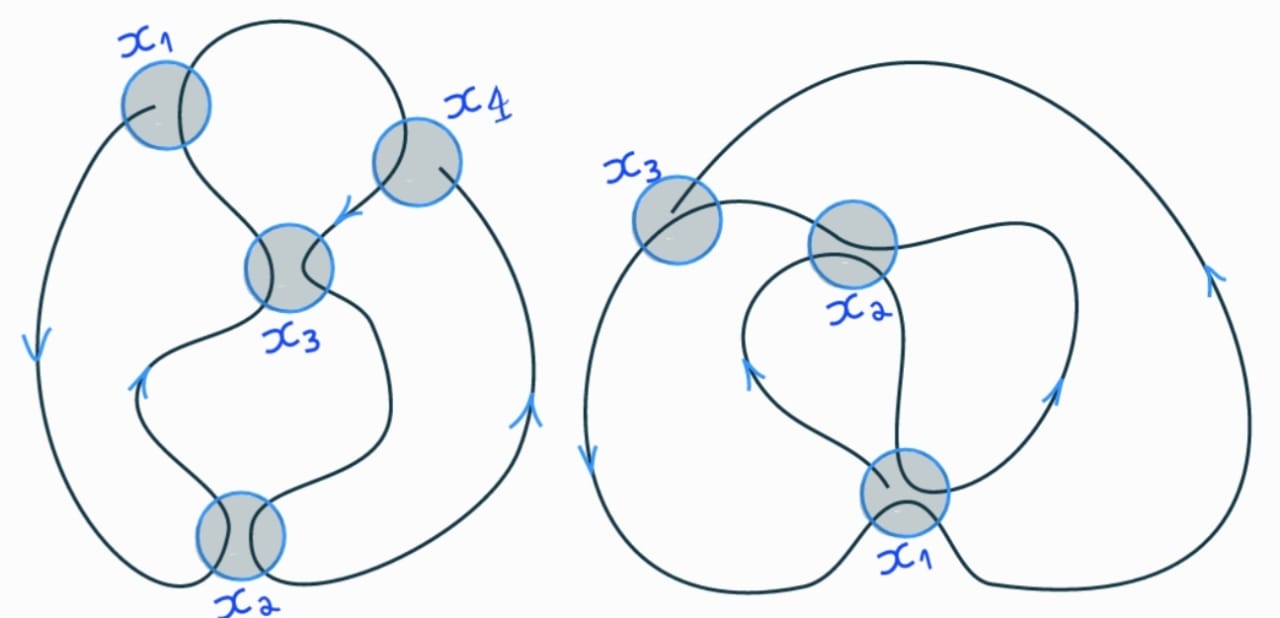}
	\caption{Links of two basic $\beta$-snakes which are not weakly outer Lipschitz equivalent but have the same spectra and the same MD-Homology. Points inside the shaded disks represent arcs with tangency order higher than $\beta$.}\label{Fig. snakes not weakly outer equiv}
\end{figure}

\begin{Lem}\label{Lem: both either nodal or segment arcs}
	Let $X$ be a $\beta$-snake and $\gamma,\gamma' \in V(X)$. If $tord(\gamma,\gamma')>\beta = itord(\gamma,\gamma')$ then both $\gamma$ and $\gamma'$ are simultaneously either segment arcs or nodal arcs.
\end{Lem}
\begin{proof}
	This is a consequence of Lemma 3.33 in \cite{CostaSouza2024}, for circular $\beta$-snakes, and its proof.  
\end{proof}
\begin{Teo}\label{Teo: simple snakes and MD-Homology}
	Let $X$ and $\tilde{X}$ be two simple $\beta$-snakes. If the following conditions simultaneously hold then they have the same MD-Homology.
	\begin{enumerate}
		\item they are weakly outer equivalent;
		\item for any two disjoint zones $Z,Z'\subset V(X)$, with $tord(Z,Z')>\beta$, which are both segments or nodal zones and their corresponding segments and nodal zones $\tilde{Z}, \tilde{Z}'\subset V(\tilde{X})$, with respect to the weakly outer homeomorphism, we have $$E_{Z,Z'}^\alpha\neq \emptyset \Leftrightarrow E_{\tilde{Z},\tilde{Z}'}^\alpha\neq \emptyset,$$ for any $\alpha\in \F$, $\alpha>\beta$. 
	\end{enumerate}
	
\end{Teo}
\begin{proof}
	Let $X$ and $\tilde X$ be two simple $\beta$-snakes weakly outer equivalent but with different MD-Homologies. In this case, there exists $b \in (\beta,\infty)$ such that $X$ and $\tilde X$, up to $b$-equivalence, have different number of generators on the respective MD-homology of degree 1. Thus, denoting by $h\colon X\rightarrow \tilde X$ the weakly outer homeomorphism between $X$ and $\tilde X$, there are $b$-equivalent arcs $\gamma,\gamma' \in V(X)$ such that $\tilde\gamma = h(\gamma)$ and $\tilde\gamma' = h(\gamma')$ are not $b$-equivalents. Once $h$ is a weakly outer homeomorphism, by Remark \ref{Rem: b-equiv and tord of arcs} and Proposition 4.10 of \cite{GabrielovSouza}, we must necessarily have $tord(\gamma,\gamma')>b\ge tord(\tilde\gamma,\tilde\gamma') >\beta$ and $itord(\gamma,\gamma') = \beta = itord(\tilde\gamma,\tilde\gamma')$. Therefore, by Lemma \ref{Lem: both either nodal or segment arcs} we may assume that both $\gamma, \gamma'$ (and consequently $\tilde\gamma,\tilde{\gamma'}$) are both segment arcs or nodal arcs. In particular, for $\alpha=tord(\gamma,\gamma')$, we would have $E_{Z,Z'}^\alpha\neq \emptyset$ but $E_{\tilde{Z},\tilde{Z}'}^\alpha = \emptyset$, where $Z,Z'$ are the segments (resp., nodal zones) containing $\gamma,\gamma'$, respectively, and $\tilde Z,\tilde Z'$ are the corresponding segments (resp., nodal zones) of $\tilde{X}$, a contradiction with condition $(2)$.
\end{proof}

\begin{Cor}\label{Cor: basic snakes and MD-Homology}
	Let $X$ and $\tilde{X}$ be basic $\beta$-snakes. If they are weakly outer equivalent and have the same spectra then they have the same MD-Homology.
\end{Cor}
\begin{proof}
	Since $X$ and $\tilde{X}$ are basic $\beta$-snakes, we could only possibly have $tord(Z,Z')>\beta$ (resp., $tord(\tilde Z,\tilde Z')>\beta$) when $Z$ and $Z'$ are distinct nodal zones of $X$ (resp., $\tilde X'$). Thus, as $X$ and $\tilde{X}$ have the same spectra, condition $(2)$ of Theorem \ref{Teo: simple snakes and MD-Homology} is satisfied. Therefore, $X$ and $\tilde{X}$ have the same MD-Homology.
\end{proof}

\begin{remark}
	Although we have determined the MD-Homology group of a basic snake from its snake name (see Corollary \ref{Cor: HMD through snake names}), the task is rather complicated when dealing with simple snakes (and consequently way worse for surfaces in general), due to the inherent combinatorics of the cluster partitions of segments containing two or more segments and the spectrum of their nodes. For example, for surfaces which link is as in the left diagram in Figure \ref{Fig. snakes not weakly outer equiv}, we will obtain different MD-Homology groups, depending on the choice of the spectrum in each node.
\end{remark}


\begin{Exam}
	The converse of Corollary \ref{Cor: basic snakes and MD-Homology} (and consequently, the one of Theorem \ref{Teo: simple snakes and MD-Homology}) is not true. Indeed, given $\alpha \in \F$, $\alpha > \beta$, let $X$ and $\tilde X$ be the basic $\beta$-snakes with spectra $\{\alpha\}$, such that $W(X) = [x_1x_2x_3x_1x_4x_3x_2x_4]$ and $W(\tilde X) = [x_1x_2x_1x_2x_3x_1x_3]$ (such snakes exists by Theorem 6.23 in \cite{GabrielovSouza}, see Figure \ref{Fig. snakes not weakly outer equiv} to visualize their links). They are not weakly outer equivalent, since they have a different number of nodes (which is equivalent to their words having a different number of distinct letters), but 
	$$MDH_{1}^{b}(X,0,d_{out}) = MDH_{1}^{b}(\tilde X,0,d_{out})= \left\{\begin{array}{cc}
		0, & \text{for } b < \beta\\
		\mathbb{Z}^{4}, & \text{for } \beta\leq b <  \alpha\\
		0, & \text{for } \alpha\leq b \\
	\end{array}\right.$$

	Moreover, condition $(2)$ of Theorem \ref{Teo: simple snakes and MD-Homology} could not be weakened. For an example, let $X$ and $\tilde{X}$ be $\beta$-snakes with link as in Figure \ref{Fig. snakes weakly outer equiv}, from the left to the right, respectively. The Valette links of those surfaces contain two segments with multiplicity $2$: the ones with adjacent nodal zones in the nodes $\mathcal{N}, \mathcal{N}'$ (say $S, S'$) and $\mathcal{\tilde{N}}, \mathcal{\tilde{N}}'$ (say $\tilde{S}, \tilde{S}'$). Let us further assume that $1\le\beta < \alpha_1 <\alpha_2$ are exponents such that  
	$$\sup\{tord(\gamma,\gamma') \mid \gamma \in S,\gamma'\in S'\} = \alpha_1 = \sup\{tord(\gamma,\gamma') \mid \gamma \in \tilde{S},\gamma'\in \tilde{S}'\},$$
	$$Spec(\mathcal{N})=Spec(\mathcal{N}')=\{\alpha_1\}; \; Spec(\mathcal{\tilde N})=Spec(\mathcal{\tilde N}') = \{\alpha_{2}\},$$
	while the spectrum of all the other nodes of $X$ and $\tilde{X}$ is $\{\alpha_{2}\}$. Therefore, $X$ and $\tilde{X}$ are weakly outer Lipschitz equivalent, but since condition $(2)$ of Theorem \ref{Teo: simple snakes and MD-Homology} is not satisfied by the corresponding nodal zones in $\mathcal{N}$ and $\mathcal{\tilde N}$ (or $\mathcal{N}'$ and $\mathcal{\tilde N}'$), they do not have the same MD-Homology. Indeed, for $\alpha_1 \le b <\alpha_{2}$, $MDH_1^b(X,0,d_{out})$ has two generators while $MDH_1^b(\tilde X,0,d_{out})$ has four.
\end{Exam}


\begin{thebibliography}{99}
	
\bibitem{birbrair1999local} 
Birbrair, L. 
{\it Local bi-Lipschitz classification of 2-dimensional semialgebraic sets.} 
Houston J. Math., vol. 25 (1999), no. 3, 453--472.

\bibitem{birbrair2008local} Birbrair, L. 
\textit{Lipschitz geometry of curves and surfaces definable in o-minimal structures}. Illinois Journal of Mathematics, vol. 52 (2008), no. 4, 1325--1353.

\bibitem{medeiros2024}
{Birbrair, L.; Denkowski, M.; Medeiros, D.L. and Sampaio, J.E.} 
{\it Universality Theorem for LNE H\"older Triangles.} 
Preprint (2024).

\bibitem{jelonek2021}
{Birbrair, L.; Fernandes, A. and Jelonek, Z.}
\textit{On the extension of bi-Lipschitz mappings.} 
Selecta Mathematica, vol. 27 (2021), article no. 15.

\bibitem{Pairsoftriangles}
{Birbrair, L. and Gabrielov, A.} {\it Lipschitz geometry of pairs of normally embedded Hölder triangles.} 
European Journal of Mathematics, vol. 8 (2022), 766--791.

\bibitem{medeiros2023}
{Birbrair, L. and Medeiros, D. L.}
\textit{Ambient Lipschitz Geometry of Normally Embedded Surface Germs.} 
Preprint arXiv:2311.18570 (2023).

\bibitem{birbrair2018arc} Birbrair, L. and Mendes, R. 
{\it Arc criterion of normal embedding.} 
Singularities and foliations. geometry, topology and applications, 549–553, Springer Proc. Math. Stat., 222, Springer, Cham, 2018. 

\bibitem{LevMendes2018} Birbrair, L. and Mendes, R. 
{\it Lipschitz contact equivalence and real analytic functions.} 
Preprint arXiv:1801.05842 (2018).

\bibitem{LBirbMosto2000NormalEmbedding} 
Birbrair, L. and Mostowski, T. {\it Normal embeddings of semialgebraic sets.} 
Michigan Math. J., vol. 47 (2000), 125--132.

\bibitem{Birbrair:1999realization} 
{Birbrair, L. and Sobolevsky, M.} 
{\it Realization of H{\"o}lder complexes.} 
Annales de la Facult{\'e} des sciences de Toulouse: Math{\'e}matiques, vol. 8 (1999), no. 1, 35--44.

\bibitem{CostaSouza2024} 
Costa, A. and Souza, E. 
{\it Lipschitz contact equivalence and real analytic functions.} 
Preprint arXiv:2101.02302 (2024).

\bibitem{Dries:1998}
van den Dries, L.
{\it Tame topology and o-minimal structures}. 
London Mathematical Society,
Lecture  Series, vol. 248, Cambridge University Press, Cambridge, 1998.

\bibitem{MDH2022}
Fern\'andez de Bobadilla, J.; Heinze, S.; Pe-Pereira, M. and Sampaio, J.~E.
{\it Moderately discontinuous homology}.
Communications on Pure and Applied Mathematics, vol. 75 (2022), no. 10, 2123-2200.


\bibitem{GabrielovSouza} Gabrielov, A. and Souza, E. 
{\it Lipschitz geometry and combinatorics of abnormal surface germs.} 
Selecta Mathematica, New Series, vol. 28 (2022), article no. 1.

\bibitem{Kurdyka92} 
Kurdyka, K. {\it On a subanalytic stratification satisfying a Whitney property with exponent 1.} 
Real algebraic geometry, Springer, (1992), 316--322

\bibitem{KurdykaOrro97} 
Kurdyka, K. and Orro, P. 
{\it Distance g\'eod\'esique sur un sous-analytique.}
Rev. Mat. Univ. Complut. Madrid, vol. 10 (1997), 173–182.

\bibitem{valette2007link} Valette, G. 
{\it The link of the germ of a semi-algebraic metric space.} 
Proc. Amer. Math. Soc., vol. 135 (2007), no.~10, 3083--3090
\end{thebibliography}
\end{document}